\documentclass[xcolor=dvipsnames]{article}
\usepackage[ngerman,english]{babel}
\usepackage{amsmath}
\usepackage{amssymb}
\usepackage{enumerate}
\usepackage{commath}
\usepackage{mdwlist}
\usepackage{mathtools}
\usepackage{ mathrsfs }
\usepackage[arrow, matrix, curve]{xy}
\usepackage[a4paper]{geometry}
\usepackage{fullpage}
\usepackage{framed}
\usepackage{tikz}
\usepackage{pgfplots}
\usetikzlibrary{arrows}
\usepackage{mdframed}
\usepackage{lipsum}
\usepackage{ziffer} 
\usepackage{mathtools}
\usepackage{verbatim}
\usepackage{xcolor}
\usepackage{float}
\usepackage{nicefrac}

\usepackage{wrapfig}
\usepackage{graphicx}
\usepackage{tabularx}
\usepackage{verbatim}
\usepackage{amsthm}
\usepackage{commath}
\usepackage{amsfonts}
\usepackage{mathrsfs}
\usepackage{amsopn}
\usepackage{stackrel}
\usepackage{manfnt}
\usepackage{bbm}
\usepackage{fancyhdr}
\usepackage{centernot}
\usepackage{stmaryrd}
\usepackage{comment}
\usepackage{cancel}
\usepackage{hyperref}
\usepackage{kantlipsum}
\allowdisplaybreaks

\theoremstyle{definition}
\newtheorem{defn}{Definition}[section]

\newmdtheoremenv{Definition}[defn]{Definition}

\newtheorem{remark}[defn]{Remark}

\newtheorem*{corol*}{Corollary}
\newtheorem*{remark*}{Bemerkung}

\newtheorem*{exmp*}{Example}

\newtheorem*{subalg*}{\textsf{Sub-Algorithmus}}

\theoremstyle{plain}
\newtheorem{thm}[defn]{Theorem}
\newmdtheoremenv{theo}[defn]{Theorem}
\newmdtheoremenv{Lemma}[defn]{Lemma}
\newmdtheoremenv{Korollar}[defn]{Corollary}
\newtheorem*{thm*}{Theorem}
\newtheorem{lemma}[defn]{Lemma}

\newtheorem{prop}[defn]{Proposition}

\newtheorem*{thmen*}{Theorem}

\newtheorem*{corolen*}{Corollary}

\numberwithin{equation}{section}

\newcommand{\R}{\ensuremath{\mathbb{R}}}

\newcommand{\N}{\ensuremath{\mathbb{N}}}

\newcommand{\E}{\ensuremath{\mathbb{E}}}

\newcommand{\ind}{\ensuremath{\mathbbm{1}}}
\newcommand{\defeq}{\ensuremath{\vcentcolon}=}
\newcommand{\eps}{\epsilon}

\newcommand{\PP}{\ensuremath{\mathbb{P}}}

\newcommand{\ddd}{\ensuremath{\text{d}}}

\newcommand{\vertiii}[1]{\vert\kern-0.25ex\vert\kern-0.25ex\vert #1 
    \vert\kern-0.25ex\vert\kern-0.25ex\vert}

\title{
How the interplay of dormancy and selection affects the wave of advance of an advantageous gene}
\author{Jochen Blath, Matthias Hammer, Dave Jacobi, Florian Nie}
\date{\today}

\begin{document}

\maketitle

\begin{abstract}
In this paper we investigate the spread of advantageous genes in two variants of the F-KPP model with dormancy. The first variant, in which dormant individuals do not move in space and instead form `localized seed banks', has recently been introduced in Blath, Hammer and Nie (2020). However, there, only a relatively crude upper bound for the critical speed of potential travelling wave solutions has been provided. The second model variant is new and describes a situation in which the dormant forms of individuals are subject to motion, while the `active' individuals remain spatially static instead. This can be motivated e.g.\ by spore dispersal of fungi, where the `dormant' spores are distributed by wind, water or insects, while the `active' fungi are locally fixed.
For both models, we establish the existence of monotone travelling wave solutions, determine the corresponding critical wave speed 
in terms of the model parameters, and characterize aspects of the asymptotic shape of the waves depending on the decay properties of the initial condition.

Interestingly, the slow-down effect of dormancy on the speed of propagation of beneficial alleles is often more serious in model variant II (the `spore model') than in variant I (the `seed bank model'), and this can be understood mathematically via probabilistic representations of solutions in terms of (two variants of) `on/off branching Brownian motions'. 
Our proofs make rather heavy use of probabilistic tools in the tradition of Watanabe (1967), McKean (1975), Bramson (1978), Neveu (1987), Lalley and Sellke (1987),  Champneys et al (1995) and others. 
However, the two-compartment nature of the model and the special forms of dormancy also pose obstacles to the classical formalism, giving rise to a variety of open research questions that we briefly discuss at the end of the paper.
\end{abstract}

\medskip

\medskip

\noindent {\em Keywords and phrases:} Fisher-Kolmogorov-Petrovski-Piscounov equation, travelling wave, dormancy, seed bank, 
on/off branching Brownian motion, advantageous gene.

\medskip

\noindent {\em MSC 2020 Subject classification:} 92D25, 60H30, 35K57.

\section{Introduction and main results}
\label{sec:intro}

\subsection{Background: The F-KPP equation with dormancy}

The F-KPP equation (named after Fisher \cite{F37} and Kolmogorov, Petrovski and Piscounov \cite{KPP37}) is  the simplest and  most prominent example of a non-linear reaction-diffusion system. In population genetics, it is used to  describe the  propagation of an `advantageous gene' or  `beneficial allele' in a bi-allelic population under the influence of directional selection. Denoting by $p(t,x)\in [0,1]$ its solution at time $t\geq 0$ and spatial position $x \in \R$, which we here interpret as the fraction of a beneficial allele present in a biological population at location $x$ and time $t$, the corresponding initial value problem is given by the non-linear second order partial differential equation
\begin{equation}
	\label{eq:F-KPP_original}
	\partial_t p(t,x) = \frac{\Delta}{2} p(t,x) + p(t,x)(1-p(t,x)), 
\end{equation}
together with a suitable initial condition 
$p_0 \in \mathcal{B}(\mathbb{R}, [0,1])$, i.e.\ a bounded Borel-measurable function defined on $\R$ taking values in $[0,1]$. It is well-known that for each
$$
\lambda  \ge  \lambda^{*, {\rm classical}} := \sqrt{2} ,
$$
this systems has a monotone travelling wave solution of speed $\lambda$, i.e.\ the solution takes the form
$$
p(t,x)=w(x-\lambda t),
$$
where $w$ is a decreasing function with 
$$
\lim_{x \to -\infty} w(x)=1 \quad \mbox{ and } \quad \lim_{x \to \infty} w(x)=0,
$$
and this solution is unique up to translations. While wave speeds above the `critical value' can be obtained by starting from initial conditions with suitable decay behaviour,  the critical wave speed $\sqrt{2}$ is realized in particular when starting in a Heaviside initial condition. These (and much finer results) can be found in the classical works \cite{McK75}, \cite{B78}, \cite{B83}, \cite{LS87}, and the model has been extended in many directions, including coupled systems with multiple components, see e.g. \cite{F91}, \cite{C97}.

\medskip

Recently, the F-KPP equation has been extended to incorporate the biological concepts of {\em dormancy} and {\em seed banks}. Here, in addition to undergoing spatial dispersal and selective pressure, individuals may independently switch into a state of low metabolic activity. In this dormant state, individuals neither move nor reproduce. The dormant individuals of a population thus form a `seed bank' which buffers genetic diversity. In biology, such a type of dormancy is often regarded as a bet-hedging strategy against unfavourable environmental conditions. The corresponding trait is widespread among many taxa, including many microbial species, see e.g. \cite{LJ11} and \cite{LHWB21+} for overviews.

For spatial systems such as the F-KPP equation, a way to incorporate a seed bank comprised of dormant individuals is to introduce a second component to the system, which describes, for each spatial position, the relative frequency of the beneficial allele in the seed bank. Active and dormant components then interact via local two-way migration (`switching'). This idea has been formalized in \cite{BHN19}, leading to the {\em coupled} two-type system
\begin{align} 
	\partial_t p(t,x) = & \frac{\Delta}{2} p(t,x)   + c (q(t,x)-p(t,x)) +s p(t,x)(1-p(t,x)), \notag\\
	\partial_t q(t,x) = & c'(p(t,x) -q(t,x)),
	\label{eq:F-KPP:with_seed_bank_original}
\end{align}
starting from a pair of initial type configurations $p_0, q_0\in \mathcal{B}(\mathbb{R}, [0,1])$ with parameters $c, c', s>0$. Here, the $p$ population represents the fraction of beneficial alleles in the active population, and $q$ describes the corresponding quantity in the dormant population. The switching rates into and out of dormancy are given by $c$ and $c'$, and $s$ describes the strength of selection favoring the beneficial allele. Note that the second component neither features a Laplace operator (a result of the fact that dormant individuals do not move) nor the selective component (dormant individuals do not feel selective pressure or competition).

\medskip
A `common source of confusion' in the context of the F-KPP equation is the use of two different conventions regarding the sign in front of the non-linearity of the equation. Historically, Fisher was interested in the advance of the {\em advantageous} gene which lead to the introduction of Equation \eqref{eq:F-KPP_original} with a positive sign (directional selection increases the frequency of the advantageous type). However, in order to exploit the fruitful probabilistic {\em method of duality} to analyze the system, it is useful  
to switch focus from the advantageous to the deleterious gene whose frequency is given by
$$u(t,x)\defeq 1-p(t,x).$$
This transformation leads to a change of sign in front of the selection term, with $u$ now solving the equation
\begin{equation}
	\label{eq:F-KPP}
	\partial_t u(t,x) = \frac{\Delta}{2} u(t,x) - u(t,x)(1-u(t,x)),
\end{equation}
which is of course equivalent to the original system.
However, note that the direction of travelling wave solutions will now be reversed (and the initial conditions need to be transformed as well).
We will follow this convention, and 
in what follows focus on the `transformed' F-KPP equation with dormancy (and selection strength $s>0$) given by
\begin{align}
	\partial_t u(t,x) = & \frac{\Delta}{2} u(t,x)   + c (v(t,x)-u(t,x)) -s u(t,x)(1-u(t,x)), \notag\\
	\partial_t v(t,x) = & c'(u(t,x) -v(t,x))
	\label{eq:F-KPP:with_seed_bank}
\end{align}
for the deleterious allele (and with corresponding initial conditions). This is in line with notation used in a majority of the probabilistic literature concerning the classical F-KPP equation (see e.g.\ \cite{H99}, \cite{McK75} or \cite{B83}).

\medskip
As mentioned above, a powerful approach to analyze solutions to the above system is via its dual Markov process (see e.g.\ \cite{S88} and \cite{AT00} for the method of duality in spatial population models). Indeed, the dual process of the (transformed) F-KPP equation \eqref{eq:F-KPP} is given by {\em branching Brownian motion} (BBM); this link was pointed out by \cite{McK75} and heavily exploited in \cite{B78}, but is already present in the more general framework of \cite{S64} and \cite{INW68a,INW68b, INW69}. Branching Brownian motion and its generalizations have been a central object of study in modern probability theory and statistical physics for several decades, see e.g.\ \cite{B15} for an overview.

\medskip

For the coupled two-type system including dormancy \eqref{eq:F-KPP:with_seed_bank},
the dual process (introduced in \cite{BHN19}) is given by so-called {\em on/off branching Brownian motion} (on/off BBM). 
This dual is again a branching Markov process (in the sense of \cite{INW68a,INW68b, INW69}), denoted by $M=(M_t)_{t \geq 0}$, formally taking values in the space 
\begin{equation}
\label{eq:statespace}
\Gamma :=\bigcup_{k \in \N_0} 
\left(\R \times \lbrace \boldsymbol{a},\boldsymbol{d}\rbrace \right)^k.
\end{equation}
Note that for fixed $k$, we interpret elements of $\left(\R \times \lbrace \boldsymbol{a},\boldsymbol{d}\rbrace \right)^k$ 
 as the spatial positions of $k$ particles in $\R$, each carrying a flag from $\lbrace \boldsymbol{a},\boldsymbol{d}\rbrace$. Particles with flag $\boldsymbol{a}$ are deemed active, while particles flagged with a $\boldsymbol{d}$ are deemed dormant.
Starting from some initial value 
$$
M_0= \big((x_1,\sigma_1), \ldots, (x_n,\sigma_n )\big)\in\left(\R \times \lbrace \boldsymbol{a},\boldsymbol{d}\rbrace \right)^n$$ 
for some $n \in \N$, the process evolves according to the following rules:
	\begin{itemize}
		\item Active particles, i.e.\ particles carrying flag $\boldsymbol{a}$, disperse in $\R$ according to independent Brownian motions
		and branch into two active particles at rate $s$. \vspace{-2mm}
		\item Independently, each active particle falls dormant at rate $c$, changing its flag from $\boldsymbol{a}$ to $\boldsymbol{d}$.\vspace{-2mm}
		\item Dormant particles neither move nor branch.\vspace{-2mm} 
		\item Independently, each dormant particle resuscitates at rate $c'$, changing its flag from $\boldsymbol{d}$ to $\boldsymbol{a}$.
	\end{itemize}
    To keep track of the number of active and dormant individuals (at each time $t\ge 0$), we denote by 
    $I_t$ and $J_t$ the (time-dependent) index sets of active and dormant particles of $M_t$, respectively. Further, we set $K_t := I_t \cup J_t$ and let $N_t:=\vert K_t \vert $ be the total number of particles at time $t\geq 0$. 
    For example, if for $t\geq 0$ we have 
    $$
    M_t=\big((M^1_t,\boldsymbol{a}), (M^2_t,\boldsymbol{d}),(M^3_t,\boldsymbol{a}),  (M^4_t, \boldsymbol{a})\big) \in \left(\R \times \lbrace \boldsymbol{a},\boldsymbol{d}\rbrace \right)^4,
    $$
then 
$$
I_t=\lbrace1,3,4 \rbrace,\, \quad  J_t=\lbrace 2 \rbrace,\, \quad \mbox{ and } \quad N_t=4.
$$
Note that this brief description can be expanded and formalized using e.g.\ the {\em Ulam-Harris labelling} (see e.g. 
\cite{N88}), but for brevity we refrain from going into the details here. 

With the help of the dual process just described, a probabilistic representation of the solution to \eqref{eq:F-KPP:with_seed_bank} can be given as follows (see \cite[Cor.\ 1.8]{BHN19}).
Indeed, starting in a Heaviside initial condition given by $$u_0(x):= v_0(x):={\bf 1}_{[0,\infty [},$$ we have the (analog of the classical) {\em McKean representation} for the solution $u$ given by 
\begin{align}\label{eq:McKean}
	u(t,x)=\PP_{(0, \boldsymbol{a})}(R_t \leq x),
	 \quad t \ge 0, \,x \in \R,
\end{align}
where $(R_t)_ {t\geq 0}$ is the position of the rightmost (maximal) particle of an on/off branching Brownian motion as defined above, started with a single active particle in $0$, i.e.\ $M_0=(0, \boldsymbol{a})$.

A question that can be answered for the classical F-KPP equation via the McKean representation concerns the speed of propagation of advantageous genes (e.g.\ when starting with reversed Heaviside initial conditions). Indeed, the position $R_t$ of the rightmost particle of the dual branching Brownian motion at time $t$  relates to the critical wave speed of the original equation via the a.s.\ equality
\begin{align} 
\label{eq: AsymptoticWaveSpeedF-KPP}
	\lim_{t \to \infty} \frac{R_t}{t}= \sqrt{2}= \lambda^{*, {\rm classical}}.
\end{align}

\medskip 

In line with intuition from biology (see e.g.\ \cite{LHWB21+}), namely that seed banks should contribute to the diversity and resilience of populations, one expects that the presence of dormancy in \eqref{eq:F-KPP:with_seed_bank} should at least slow down the speed of travelling wave solutions, if not preventing their emergence entirely. Following this line of thought, a first relatively crude argument 
in \cite{BHN19} shows that the speed of the rightmost particle is indeed reduced significantly due to the effect of dormancy. More precisely, via a simple first moment bound in combination with a many-to-one lemma, the upper bound
$$
\lim_{t \to \infty} \frac{R_t}{t} \le \sqrt{\sqrt{5}-1}\approx 1{.}11 < \sqrt{2}
$$
has been obtained (for the case $c=c'=s=1$). However, it was left open in \cite{BHN19} whether this bound is actually sharp, and even whether non-trivial travelling wave solutions exist at all.

\medskip

In the present paper, our first aim is to extend this result by providing the existence of travelling wave solutions, the exact value of the critical wave speed of the F-KPP equation with dormancy and the asymptotic speed (up to first order) of the rightmost particle of the on/off branching Brownian motion. 
In the case $c=c'=s=1$, we will see that the previous upper bound is {\em not} sharp, and that the correct value is given by 
$$
\lambda^* \approx 0{.}98 < \sqrt{\sqrt{5}-1} <  \sqrt{2}.
$$
Our second aim in this paper is to  expand our modelling approach. Indeed, there is another natural way to incorporate seed banks into an F-KPP based model, that was not considered in \cite{BHN19}, and for which we aim to obtain similar results while also investigating their quantitative differences.

\subsection{Two models for the interplay of dormancy and dispersal}

Understanding the role of different forms of dormancy in population genetics and ecology is currently an active field of research, 
and this  holds in particular for the interplay between dormancy and spatial dispersal (\cite{WLL19}, \cite{BC14}, \cite{GHO20}). While it is usually assumed that there is a trade-off between dormancy and dispersal (dormancy preventing dispersal), there are also situations where dormancy actually facilitates dispersal. Natural examples include fungi, whose dormant life-stages are spores which are often dispersed by wind. Their robust dormant form allows them to potentially travel far distances. Other mechanisms of dispersal of spores include  transmission via insects, or by water. Common to these examples is that the `active' reproducing form (fungus) remains locally static, while the dormant form (spore) disperses. To capture such a scenario in an F-KPP based system, it seems natural to move the diffusion operator from the first (active) component to the second (dormant) component.  In what follows, we thus present two variants of our F-KPP model with dormancy, describing either dispersal of actives only (`seed bank model'), or dispersal of dormants only (`spore model'). While the first variant is merely a  small extension of the system \eqref{eq:F-KPP:with_seed_bank} to more general selection terms, the second variant appears to be new.

\begin{defn}[F-KPP equation with dormancy, variant I: seed bank model]
The initial value problem associated with the {\em F-KPP equation with dormancy, variant I}, is given by the coupled system
\begin{align} \label{eq:FKPPDormancy}
    \partial_t u(t,x) &= \frac{\Delta}{2} u(t,x)+c (v(t,x)-u(t,x)) +\kappa s(u(t,x)) u(t,x)(u(t,x)-1),\nonumber\\
    \partial_t v(t,x) &= c' (u(t,x)-v(t,x)),\qquad t>0,\; x\in\R
\end{align}
with initial conditions $u_0, v_0\in \mathcal{B}(\mathbb{R}, [0,1])$ and parameters $c, c',\kappa > 0$,
where the {\em selection term} is of the form
\begin{align*}
    s(u)u(u-1)=  \sum_{k=1}^\infty p_k (u^{k+1}-u), \quad u \in [0,1],
\end{align*}
for a given probability distribution $p=(p_k)_{k \in \N}\subseteq [0,1]^\N$ such that $\sum_{k=1}^\infty p_k=1 $ and $\sum_{k=1}^\infty p_k k < \infty$.
\end{defn}

Note that the form of the selection term ensures the duality to  {\em on/off branching Brownian motion} with general (not necessarily binary) branching mechanism. 
In fact, this branching mechanism is precisely given by the probability distribution $(p_k)_{k\in\N}$, i.e., the probability to see $k+1$ offspring in a branching event (which happen with overall rate $\kappa$) is precisely $p_k$. 
Also note that Equation  \eqref{eq:F-KPP:with_seed_bank} is the special case of \eqref{eq:FKPPDormancy} with $p_1=1$ and $\kappa s(u)\equiv s$.

The second variant is distinguished from the previous one by moving the Laplacian from the first to the second component:

\begin{defn}[F-KPP equation with dormancy, variant II: spore model]
The initial value problem associated with the {\em F-KPP equation with dormancy, variant II} is given by the coupled system 
\begin{align} \label{eq:FKPPDormancyII}
    \partial_t  \tilde  u(t,x) &= \tilde c ( \tilde  v(t,x)- \tilde  u(t,x)) + \kappa \tilde s( \tilde  u(t,x))  \tilde  u(t,x)( \tilde  u(t,x)-1), \nonumber\\
    \partial_t  \tilde  v(t,x) &= \frac{\Delta}{2}  \tilde  v(t,x)+ \tilde c' ( \tilde  u(t,x)- \tilde  v(t,x)), \qquad t>0,\; x\in\R
\end{align}
with initial conditions $\tilde u_0, \tilde v_0\in \mathcal{B}(\mathbb{R}, [0,1])$ and parameters $\tilde c, \tilde c', \kappa > 0$, 
where the {\em selection term} is of the form
\begin{align*}
    \tilde s( \tilde  u) \tilde  u( \tilde  u-1)= \sum_{k=1}^\infty \tilde  p_k ( \tilde  u^{k+1}- \tilde  u), \quad  \tilde  u \in [0,1],
\end{align*}
for a given probability distribution $\tilde p=(\tilde p_k)_{k \in \N}\subseteq [0,1]^\N$ such that $\sum_{k=1}^\infty \tilde  p_k=1 $ and $\sum_{k=1}^\infty \tilde  p_k k < \infty.$
\end{defn}

\begin{remark}[Alternative interpretation of model II]
    Note that in our interpretation, the selective pressure is always applied to the active population (e.g.\ corresponding to fertility selection). An equivalent interpretation of model variant II can be obtained by interchanging the roles of $\tilde u$ and $\tilde v$: Now, the interpretation is that active individuals disperse, but are not subject to selective pressure, which acts only on dormant individuals.
\end{remark}

\begin{remark}[Relation to existing theory for coupled reaction-diffusion systems]
Note that both model variants I and II can be considered as coupled reaction-diffusion systems. Coupled systems have been considered e.g.\ by Freidlin \cite{F91},  Champneys et.\ al.\ \cite{C97} and by Bovier and Hartung \cite{BH23}. However, there the assumption is that the Laplace operator is present in both sub-populations, which is not the case in our set-up. Still, several of the arguments of Champneys et.\ al. can be employed for the analysis of our system, as we will see later in the paper. 
\end{remark}

\begin{remark}[Delay representations]
    Both model variants allow a reformulation as delay-equations. While the delay in the first model is essentially the same as in the model considered in \cite{BHN19}, simply including a more general selection term, the delay formulation of the second variant can in general only be provided implicitly. The reason is that if the Laplace operator is moved to the second component, it is the first component that can be represented as solution of an ODE. However, since the selection term is of second order (or higher), one  ends up with a Riccati-type equation, which in general cannot be solved explicitly. Since the delay-representation is not required for our results (in contrast to the situation in \cite{BHN19}), we refrain from going into the details here. 
\end{remark}

\medskip

Both of the above model variants again have dual spatial branching processes. In order to provide the corresponding dualities, we first formally introduce the  two  variants of on/off branching Brownian motions (on/off BBM, variant I and II) arising as duals to \eqref{eq:FKPPDormancy} and \eqref{eq:FKPPDormancyII}. We employ the notation from the preceding section.

\begin{defn}[on/off-BBM, variant I] \label{def:VariantI}
    On/off branching Brownian motion variant I  corresponds to the system \eqref{eq:FKPPDormancy}  and is the unique Markov process $(M_t)_{t \ge 0}$ with state space $\Gamma$ from \eqref{eq:statespace} evolving according to the following rules:
	\begin{itemize}
		\item Active particles (carrying flag $\boldsymbol{a}$) disperse in $\R$ according to independent Brownian motions. \vspace{-2mm}
		\item Active particles branch at rate $\kappa$ into $k+1 \in \N$ offspring particles according to the distribution $p=(p_k)_{k\in \N}.$ \vspace{-2mm}
		\item Independently, each active particle falls dormant at rate $c$, changing its flag from $\boldsymbol{a}$ to $\boldsymbol{d}$.\vspace{-2mm}
		\item Dormant particles (carrying flag $\boldsymbol{d}$) neither move nor branch.\vspace{-2mm} 
		\item Independently, each dormant particle resuscitates at rate $c'$, changing its flag from $\boldsymbol{d}$ to $\boldsymbol{a}$.
	\end{itemize}
\end{defn}
	
\begin{defn}[on/off-BBM, variant II]\label{def:VariantII}
    On/off branching Brownian motion variant II corresponds to the system \eqref{eq:FKPPDormancyII} and is the unique Markov process $(\tilde M_t)_{t \ge 0}$ with state space $\Gamma$ from \eqref{eq:statespace}  evolving according to the following rules:
	\begin{itemize}
		\item Active particles (carrying flag $\boldsymbol{a}$) do \textit{not} move  but  branch at rate $\kappa$ into $k+1 \in \N$ offspring particles according to the distribution $\tilde p=(\tilde p_k)_{k\in \N}$.\vspace{-2mm} 
		\item Independently, each active particle falls dormant at rate $\tilde c$, changing its flag from $\boldsymbol{a}$ to $\boldsymbol{d}$.\vspace{-2mm}
		\item Dormant particles (carrying flag $\boldsymbol{d}$) disperse in $\R$ according to independent Brownian motions. \vspace{-2mm}
		\item Dormant particles do \textit{not} reproduce. \vspace{-2mm}		
		\item Independently, each dormant particle resuscitates at rate $\tilde c'$, changing its flag from $\boldsymbol{d}$ to $\boldsymbol{a}$.
	\end{itemize}
\end{defn}

\noindent
As before, we denote by $I_t$ resp.\ $J_t$ the index sets of active resp.\ dormant particles at time $t\ge0$.
\medskip

With this notation, we are now in a position to provide a formal statement of the duality between on/off BBM and the F-KPP equation with dormancy 
in each of the variants I and II.

\begin{prop}\label{prop:duality}
Consider Equation \eqref{eq:FKPPDormancy} resp.\ \eqref{eq:FKPPDormancyII} with initial conditions $u_0,v_0 \in \mathcal{B}(\mathbb{R}, [0,1])$ resp.\ $\tilde u_0, \tilde v_0 \in \mathcal{B}(\mathbb{R}, [0,1])$. Moreover, let $(M_t)_{t\geq 0}$ resp.\ $(\tilde M_t)_{t \geq 0}$ be on/off BBMs of variant I resp.\ variant II. Then, Equation \eqref{eq:FKPPDormancy} has a unique solution taking values in $[0,1]$ which is given by 
\begin{align*}
    u(t,x) &=\E_{(x,\boldsymbol{a})}\left[\prod _{\alpha \in I_t}u_0(M^\alpha_t)\prod _{\beta \in J_t}v_0(M^\beta_t)\right], 
	& v(t,x) =\E_{(x,\boldsymbol{d})}\left[\prod _{\alpha \in I_t}u_0(M^\alpha_t)\prod _{\beta \in J_t}v_0(M^\beta_t)\right], 
\end{align*}
and the analogous statement holds for Equation \eqref{eq:FKPPDormancyII} with $u, v, M$ replaced by $\tilde u, \tilde v, \tilde M$. 
\end{prop}

For variant I, Proposition \ref{prop:duality} is a small extension (to more general selection terms resp.\ branching mechanisms) of the corresponding result in \cite{BHN19}. We remark however that actually for both models, the probabilistic representation of the corresponding solutions 
is already contained in the very general framework of \cite{INW69}, since both variants of on/off BBM are branching Markov processes in the sense of that paper. We will give a brief overview of how the two models fit into this framework in Section \ref{sec:add_mart-1}.

\subsection{The linearized systems and their `wave speed functions'}
\label{ssn:travel_waves}

We aim to establish, for both model variants I and II, critical wave speeds $\lambda^*, \tilde \lambda^* \in\, ]0, \infty[$ 
and the existence of monotone travelling wave solutions 
\begin{align*}
   u(t,x) =f(x-\lambda t), 
   \qquad
   v(t,x) =g(x-\lambda t)
\end{align*}
to Equation \eqref{eq:FKPPDormancy} for all speeds $\lambda > \lambda^*$, 
resp.\ to Equation \eqref{eq:FKPPDormancyII} for all speeds $\tilde \lambda > \tilde \lambda^*$. 
For variant I, this is equivalent to the pair $(f,g)$ solving the system 
\begin{align} \label{eq:TravellingWaveI}
    0=\frac{1}{2} f''+\lambda f' +c(g-f)+ \kappa s(f) f(f-1), \quad  \quad\quad  0=\lambda g'+c'( f- g),
\end{align}
which we call the \textit{travelling wave equation} for variant I (with a similar system for variant II). 
Solutions to this equation will also be called \emph{travelling waves of speed $\lambda$.}

\medskip

Following classical ideas (cf.\ \cite[p.\ 83]{C97}), 
	we first consider the linearized versions of the corresponding F-KPP equations with dormancy, which for  variant I  is given by 
\begin{align} \label{eq:LinearFKPPDormancy}
    \partial_t u(t,x) &= \frac{\Delta}{2} u(t,x)+c (v(t,x)-u(t,x)) + {\tt s} u(t,x),\nonumber\\
    \partial_t v(t,x) &= c' (u(t,x)-v(t,x))
\end{align}
with 
\begin{equation}\label{eq:def-s}
 {\tt s}:= \kappa \sum_{k=1}^\infty p_k k
 \end{equation}
(the corresponding equation for variant II is again omitted for brevity). Note that all information on the branching mechanism encoded in the selection term here is condensed into the real number ${\tt s}$, which is given by the overall branching rate $\kappa$ times the expected number of offspring (excluding the parent). 
Equivalently, the corresponding linearized travelling wave 
equation is now given by the system
\begin{align} \label{eq:LinearTravellingWAVE}
    0=\frac{1}{2} {\tt f}''+\lambda {\tt f}' +c({\tt g}-{\tt f})+ {\tt s} \tt f, \quad  \quad\quad  0=\lambda {\tt g'} +c'(\tt f-\tt g).
\end{align}
It will be suitable to interpret this linear system (as in \cite{C97})  
as a vector-valued equation given by 
\begin{align}
    0=\frac{1}{2}A \vec{w}'' + \lambda \vec{w}' + Q \vec{w} + R\vec{w},
\end{align}
where 
\begin{align*}
    A=\begin{pmatrix}
1 & 0\\
0&0
\end{pmatrix},\quad 
Q=\begin{pmatrix}
-c & c\\
c'&-c'
\end{pmatrix},\quad 
R=\begin{pmatrix}
{\tt s} & 0\\
0&0
\end{pmatrix}
\quad \mbox{ and }   \quad   \vec{w}(x)=\begin{pmatrix}
{\tt f}(x)\\
{\tt g}(x)
\end{pmatrix}.
\end{align*}

\noindent The usual Ansatz to solve this system is to choose $(\tt f,\tt g)$ in dependence on a decay parameter 
$\mu < 0$ 
of the form 
\begin{align}
\label{eq:f,g-specific}
    {\tt f}_\mu (x) = d_1(\mu) e^{\mu x} \quad \mbox{ and } \quad
   {\tt g}_\mu (x) = d_2(\mu) e^{\mu x},\qquad x\in\R.
\end{align}
Writing $\vec{d}(\mu):=(d_1(\mu),d_2(\mu))^T$, this leads to the eigenvalue problem
\begin{align}
\label{eq:eigenvalue_problem}
    \left(\frac{1}{2}\mu^2 A +Q +R \right) \vec{d}(\mu)=-\mu \lambda \vec{d}(\mu).
\end{align}
For each given $\mu$, solving this for $\lambda$ gives two possible values, namely
\begin{align*}
 \lambda^-_\mu = - \frac{1}{2 \mu} \Big(  {\tt s}-c'-c - \sqrt{c^2+2\,c\,c'-c\,\mu^2-2\,c\, {\tt s}+ ({c'})^2+c'\,\mu^2+2\,c'\, {\tt s}+\frac{\mu^4}{4}+\mu^2\, {\tt s}+{\tt s}^2} + \frac{\mu^2}{2}\Big)  \end{align*}
 and
 \begin{align}\label{eq:lambda_mu}
 \lambda^+_\mu
 = - \frac{1}{2 \mu} \Big( {\tt s}-c'-c + \sqrt{c^2+2\,c\,c'-c\,\mu^2-2\,c\,{\tt s}+({c'})^2+c'\,\mu^2+2\,c'\,{\tt s}+\frac{\mu^4}{4}+\mu^2\,{\tt s}+{\tt s}^2} + \frac{\mu^2}{2}\Big) .
 \end{align}

From Perron-Frobenius-Seneta theory \cite{S1981} it is not hard to see that for each $\mu < 0$, the eigenvalue $-\mu \lambda^+_\mu$ for \eqref{eq:eigenvalue_problem} is strictly positive, and consequently that $\mu\mapsto\lambda^+_\mu$ is a positive function. Moreover, 
the corresponding eigenvector $\vec{d}(\mu)$ (which is unique up to constant multiples) can be computed explicitly, and in particular can be chosen with strictly positive entries. We refer to Lemma \ref{lem:speedfunctionpositivity} 
in the Appendix for details. 
Since these properties will be essential ingredients in the proofs below, we will now mainly
 focus on $\lambda^+_\mu$ and omit the superscript ${}^+$ in what follows, i.e.\ we simply write $\lambda_\mu$ for the value in \eqref{eq:lambda_mu}.
Further, the notation $\vec{d}(\mu)$ will henceforth always denote 
the unique eigenvector in \eqref{eq:eigenvalue_problem} satisfying 
\begin{equation}\label{eq:eigenvector}
1=d_1(\mu)>d_2(\mu)>0,
\end{equation}
 see Lemma \ref{lem:speedfunctionpositivity}.

\begin{prop}[Speed function for travelling wave solutions, model variant I]
\label{prop:speed_function}
The differentiable map 
$$
\lambda_\bullet:\,  ]-\infty, 0[ \, \to \, ]0, \infty [ \, , \quad \mu \mapsto \lambda_\mu,
$$
called the {\em speed function} of the linearized travelling wave equation for model variant I, 
has a unique local and global minimum 
on the negative half axis. The minimizer $\mu^*\in\,]0,\infty[$ is called {\em critical decay rate, and the corresponding minimal value} 
$$
\lambda^*:=\lambda_{\mu^*}>0
$$
is called {\em critical wave speed}. 
\end{prop}

Again, this is not hard to check and we refer to the Appendix for a proof. 
 For model variant II, we obtain for the linearized system the matrices
\begin{align*}
    \tilde A=\begin{pmatrix}
0 & 0\\
0&1
\end{pmatrix},\quad 
\tilde Q=\begin{pmatrix}
-\tilde c & \tilde c\\
\tilde c'&-\tilde c'
\end{pmatrix},\quad 
\tilde R=\begin{pmatrix}
\tilde{\tt s} & 0\\
0&0
\end{pmatrix}
\end{align*}
and similar arguments as above lead to the positive solution
\begin{align*}
\tilde \lambda_{\tilde\mu}:=\tilde \lambda_{\tilde\mu}^+ &= -\frac{1}{2\tilde \mu}\left(\tilde{\tt s}-\tilde c'-\tilde c+\sqrt{{\tilde c}^2+2\, \tilde c\,\tilde c'+\tilde c\,{\tilde \mu}^2-2\,\tilde c\,\tilde{\tt s}+(\tilde c')^2-\tilde c'\,{\tilde \mu}^2+2\,\tilde c'\,\tilde{\tt s}+\frac{{\tilde\mu}^4}{4}-{\tilde \mu}^2\,\tilde{\tt s}+\tilde{\tt s}^2}+\frac{{\tilde \mu}^2}{2}\right)
\end{align*}
of the corresponding eigenvalue problem. We define $ \tilde \lambda^*$ and $\tilde \mu^*$ as the quantities from Proposition \ref{prop:speed_function} (i.e.\ as the global minimum resp.\ minimizer) corresponding to the speed function of variant II. They satisfy analogous properties and results, and the details are again omitted for brevity. 

\medskip

Below we will see that the minimal values of the speed functions, $\lambda^*=\lambda_{\mu^*}$ and ${\tilde \lambda}^*=\tilde \lambda_{{\tilde\mu}^*}$, indeed provide the critical wave speeds of monotone travelling wave solutions for model variants I and II. Figure \ref{fig:speed_functions} depicts these speed functions, and their minimizers, for parameters $c=\tilde c= c'= \tilde c'= {\tt s}= \tilde{\tt s}=1$ in models I and II, and also for the classical F-KPP equation. 
Their dependence on the parameter values in our models will be further discussed in Section \ref{sec:discussion}.

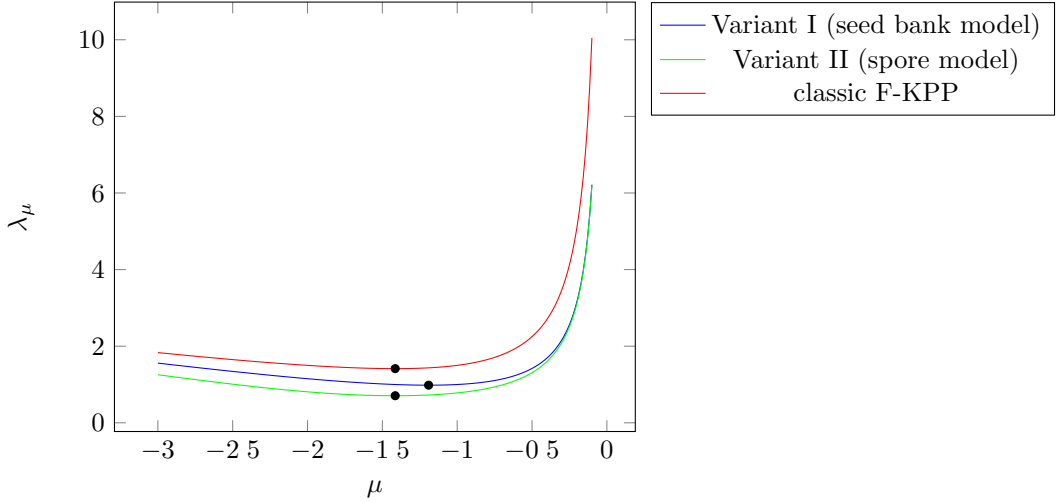
\begin{figure}[htbp]
    \centering
\begin{tikzpicture}
\begin{axis}[legend pos=outer north east, xlabel=$\mu$, ylabel=$\lambda_\mu$]

\addplot[blue,domain=-3:-0.1, samples=201] {-(-2 + x^2 + sqrt(20 + 4*x^2 + x^4))/(4*x)};
\draw[fill] (axis cs:{-1.19103,0.982416}) circle [radius=1.5pt] node[above left] {$ $};
\addlegendentry{Variant I (seed bank model)}
\addplot[green,domain=-3:-0.1, samples=201] {-(-2 + x^2 + sqrt(20 - 4*x^2 + x^4))/(4*x)};
\draw[fill] (axis cs:{-1.41421,1/sqrt(2)}) circle [radius=1.5pt] node[above left] {$ $};
\addlegendentry{Variant II (spore model)}
\addplot[red,domain=-3:-0.1, samples=201] {-(0.5*x + 1/x)};
\draw[fill] (axis cs:{-1.41421,sqrt(2)}) circle [radius=1.5pt] node[above left] {$ $};
\addlegendentry{classic F-KPP}
\end{axis}
\end{tikzpicture}
    \caption{Comparison of the speed functions for the classical F-KPP equation, and model variants I and II (for parameters $c=\tilde c=c'=\tilde c'={\tt s}=\tilde{\tt s}=1$). The dots indicate the position and size of the minimum of the respective speed function.}
\label{fig:speed_functions}
\end{figure}

\subsection{Main results}
\label{ssn:main_results}

We are now in a position to state our main results. Their proofs can be found in Sections \ref{sec:add_mart} and \ref{sec:travelling}.  We begin with results on the speed of the rightmost particle of on/off BBM, for both model variants I and II.
Recall the definition of $\lambda^*$, $\mu^*$, $ \tilde \lambda^*$ and $\tilde \mu^*$ from the previous section.

\begin{thm}[Speed of rightmost particles]
\label{thm:rightmost}
Let $(R_t)_ {t\geq 0}$ and  $(\tilde R_t)_ {t\geq 0}$ be the stochastic processes describing the position of the rightmost particle of on/off branching Brownian motion $(M_t)_{t \geq 0}$ dual to model variant I, and $(\tilde M_t)_{t \geq 0}$ dual to model variant II, each started with a single (active or dormant) particle. 
Then we have, almost surely,
\begin{align*}
\lim_{t \to \infty} \frac{R_t}{t} = \lambda^* \quad \mbox{ and } \quad 
\lim_{t \to \infty} \frac{\tilde R_t}{t} = \tilde \lambda^*.
\end{align*}
\end{thm}

Using the McKean representation \eqref{eq:McKean}, the position of the rightmost particle gives information about the solution of the dual F-KPP equations. While this already establishes a `speed' of propagation when started in a Heaviside initial condition (see Theorem \ref{thm:critical_speed} below), the result can be significantly strengthened.

\begin{thm}[Existence of travelling wave solutions]
\label{thm:existence_of_supercritical_waves}
For each $\mu \in\, ]\mu^*, 0[$, there exists a
 solution $(f_\mu,g_\mu)$ to the travelling wave equation \eqref{eq:TravellingWaveI} with speed $\lambda=\lambda_\mu$ such that $f_\mu$ and $ g_\mu$ are increasing from $0$ to $1$.
In particular, there exist travelling wave solutions to Equation \eqref{eq:FKPPDormancy} for all speeds $\lambda>\lambda^*$. The analogous statement  holds for model variant II, i.e.\ for the system \eqref{eq:FKPPDormancyII}, in terms of $\tilde \mu^*$ and $\tilde \lambda^*$.
\end{thm}

Regarding the shape of the travelling waves 
in this `supercritical case', that is, for wave speeds strictly greater than the critical speed, we obtain the following asymptotic decay result.
Recall that $({\tt f}_\mu,{\tt g}_\mu)$, defined in \eqref{eq:f,g-specific}, solves the linearized travelling wave equation \eqref{eq:LinearTravellingWAVE} with speed $\lambda=\lambda_\mu$.

\begin{prop}[Asymptotic decay of travelling waves] 
\label{prop:supercritical_travelling_wave_asymptotics}
Let $\mu \in\, ]\mu^*, 0[$. Then the corresponding travelling wave 
$( f_\mu,  g_\mu)$ 
from Theorem \ref{thm:existence_of_supercritical_waves} satisfies the asymptotic relationship
\begin{align*}
	(1- f_\mu (x)) \sim {\tt f}_\mu(x) \quad \mbox{ and } \quad
	(1-g_\mu (x)) \sim {\tt g}_\mu(x)
\end{align*}
as $x \to \infty$.
The analogous statement holds for model variant II in terms of $\tilde \mu^*$.
\end{prop}

The asymptotic shape of the travelling waves is thus described by the solutions to the linearized systems. In fact, depending on the decay rate of the initial condition, in the supercritical case all solutions converge to the corresponding travelling waves.

\begin{thm}[Convergence of solutions in the supercritical regime]
\label{thm:supercritical_travelling_wave_asymptotics}
Consider the solution to system \eqref{eq:FKPPDormancy} with 
initial conditions $u_0,v_0\in \mathcal{B}(\mathbb{R}, [0,1])$. 
Then, for any $\mu \in\, ]\mu^*, 0[$ we have that if
\begin{align*}
    1-u_0(x) \sim  {\tt f}_\mu(x) \quad \mbox{ and } \quad
    1-v_0(x) \sim  {\tt g}_\mu(x)\quad \mbox{ as } x \to \infty,
\end{align*}
then for all $x\in \R$
\begin{align*}
    u(t,x+\lambda_\mu t) \to  f_\mu(x) \quad \mbox{ and } \quad
    v(t,x+\lambda_\mu t) \to  g_\mu(x) \quad \mbox{ as } t \to \infty.
\end{align*}
The analogous statement holds for the system \eqref{eq:FKPPDormancyII} in terms of $\tilde \mu^*$.
\end{thm}

As for the classcial F-KPP equation, there are no monotone travelling wave solutions with speed strictly below $\lambda^*$ (subcritical case).

\begin{thm}
\label{thm:no_subcritical_monotone_waves}
There are no travelling wave solutions for system \eqref{eq:FKPPDormancy} 
 resp.\  \eqref{eq:FKPPDormancyII} increasing from $0$ to $1$ with speeds $0 \leq \lambda < \lambda^*$ resp. $0 \leq \tilde \lambda < \tilde \lambda^*$.
\end{thm}

The speed of propagation of the advantageous allele in systems with dormancy, when started in Heaviside initial conditions, falls into the critical regime. For this important case, we have the following result.

\begin{thm}[Speed of propagation of the beneficial allele in the critical regime]
\label{thm:critical_speed}
Consider the solution to Equation \eqref{eq:FKPPDormancy} with Heaviside initial conditions $u_0=v_0=\ind_{\R^+ }$. 
Let $x\in \R$.
Then, 
for all $\lambda>\lambda^*$ we have
\begin{align*}
    u(t,x+\lambda t) \to 1 \quad \mbox{ and } \quad 
    v(t,x+\lambda t) \to 1
\end{align*}
as $t \to \infty$, and 
for all $\lambda<\lambda^*$ we have
\begin{align*}
    u(t,x+\lambda t)\to 0 \quad \mbox{ and } \quad 
    v(t,x+\lambda t)\to 0
\end{align*}
as $t \to \infty$. The analogous statement holds for the system \eqref{eq:FKPPDormancyII} in terms of $\tilde \lambda^*$.
\end{thm}

Again we see that the beneficial allele propagates at a linear speed, which can be computed explicitly, depending on the model parameters.  However, we currently lack finer results on travelling wave solutions. Some of the reasons and difficulties will be discussed in Section \ref{sec:discussion}.

\begin{remark}
In simulations one may see the emergence of approximate travelling waves when the system is started from Heaviside initial conditions (see Figure \ref{fig:1}). 
Theorem \ref{thm:critical_speed} 
establishes that the bulk of the mass of the approximate waves cannot move faster or slower than $\lambda^*t$ for large $t$ which is why we speak of approximate travelling waves of speed $\lambda^*$.
\end{remark}

\begin{figure}[h]
\centering
	\includegraphics[scale=0.3]{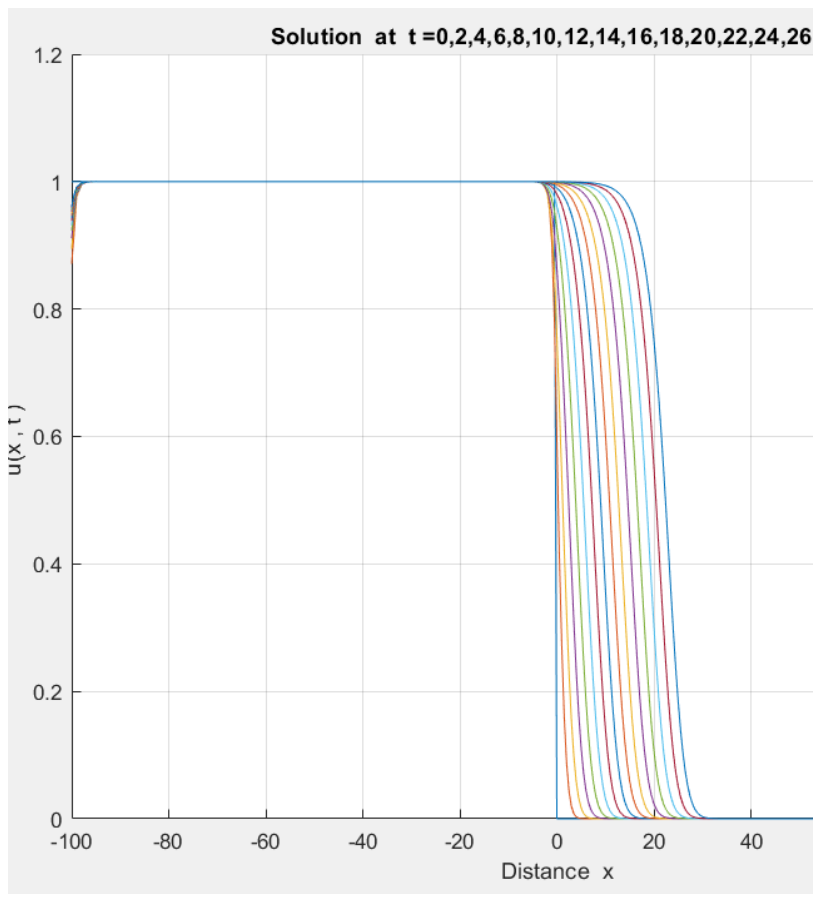}
	\includegraphics[scale=0.3]{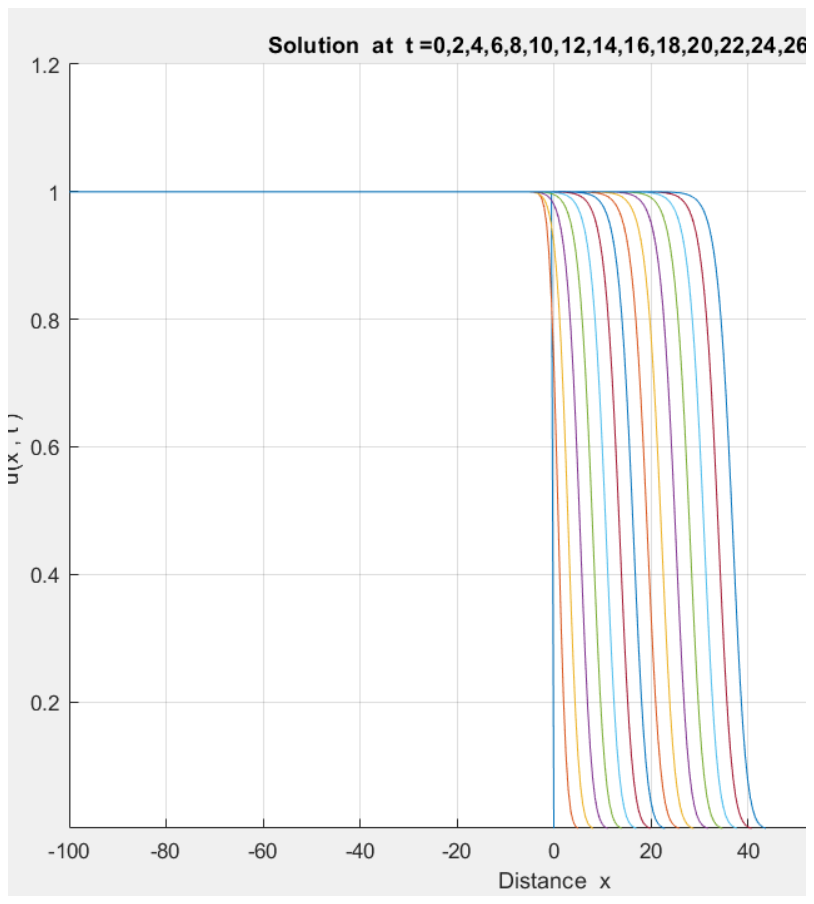}
	\vspace{-2cm}
	\caption{Active component $(1-u)$ of the F-KPP equation with seed bank (left) compared to the classical F-KPP equation (right), started from reversed Heaviside initial conditions. The position of the wave-front is drawn at the same consecutive time-points each, indicating that the presence of the seed bank significantly reduces the speed of propagation of the wave front. }
\label{fig:1}
\end{figure}

While both model variants so far exhibit the same qualitative behaviour, it is certainly interesting to investigate the quantitative differences. Unfortunately, due to the rather implicit description of the speed functions $\lambda_\bullet$ and $\tilde \lambda_\bullet$, 
it is not easy to achieve general analytic results. The following is an example of what can be observed for fixed sets of parameters. However, all quantities are readily accessible via simulation. 

Note that the result below already shows that for `unit parameters', dormancy slows the speed of propagation of beneficial alleles more severely in the spore model than in the seed bank model, and both significantly reduce the spread of beneficial alleles in comparison to the classical F-KPP model. The picture emerging for other parameter choices is quite rich, and a more detailed discussion with concrete values using numerical methods will be provided in Section \ref{sec:discussion}.

\begin{prop}
\label{prop:relation}
For $c=\tilde c=c'=\tilde c'={\tt s}=\tilde{\tt s}=1$, we have that $\tilde \lambda_\mu<  
\lambda_\mu 
< \lambda^{\rm \tiny  classical}_\mu$ for each $\mu<0$, 
and in particular that $\tilde \lambda^* < \lambda ^* < \lambda^{*, {\rm \tiny classical}}$.
\end{prop}

This situation is depicted in Figure \ref{fig:speed_functions}. The concrete values can be computed numerically (where necessary) and yield
$$
\tilde \lambda^* =\frac{1}{\sqrt{2}}  , \quad \lambda ^* \approx 0{.}98, \quad \lambda^{*, {\rm \tiny classical}}= \sqrt{2}.
$$

The proof is straightforward, see again the Appendix for details.

\subsection{Organization of the paper}

The proofs for the derivation of the critical wave speed and the existence of a corresponding travelling wave solution follow well-trodden yet elegant paths and employ a convergence analysis of suitable additive and multiplicative martingales. 

Indeed, in Section \ref{sec:add_mart}, we use the solutions to the linearized wave equation 
\eqref{eq:LinearTravellingWAVE} to construct a corresponding additive martingale based on the dual on/off branching Brownian motion. 
The convergence properties of this martingale then provide information about the asymptotic speed of the rightmost particle in the on/off branching Brownian motion. 

In Section \ref{sec:travelling}, we establish the existence and properties of travelling wave solutions.
This is done again following a well-known general recipe by characterizing solutions to the travelling wave equation 
 \eqref{eq:TravellingWaveI} 
in terms of suitable multiplicative martingales.

We complete the paper by a discussion (Section \ref{sec:discussion}) of the  wave speed in all model variants in dependence on the underlying parameters, by addressing open problems related in particular to the critical case, and by outlining possible future research. 

Finally, the Appendix (Section \ref{sec:Appendix}) establishes technical results on the speed function which are needed in the previous sections.

\section{Additive Martingales and the speed of the rightmost particle in on/off BBM}
\label{sec:add_mart}

In this section, we focus on model variant I. Analogous results and proofs can be obtained for variant II, but will be skipped for brevity. We employ the classical martingale approach 
pioneered by Watanabe \cite{Wat67} and later used by McKean \cite{McK75} and Neveu \cite{N88}, among others. 

\subsection{The additive martingale $(X_t^{\lambda_\mu})_{t\ge0}$}\label{sec:add_mart-1}

For $\mu < 0$, let $\tt (f_\mu,g_\mu)$ be given as in Equation \eqref{eq:f,g-specific} with decay rate $\mu$.
We recall that $\tt (f_\mu,g_\mu)$ solves the linearized travelling wave equation \eqref{eq:LinearTravellingWAVE} with speed $\lambda=\lambda_\mu>0$,
where $\lambda_\mu$ is the value of the speed function from Proposition \ref{prop:speed_function} at $\mu$. 
Finally, let $(M_t)_{t \ge 0}$ be an on/off BBM as defined in Definition \ref{def:VariantI} (we use all notations introduced there) and $(\mathcal{F}_t)_{t \ge 0}$ 
its canonical filtration. 
The following observation is key to the results in this section.

\begin{prop}
\label{prop:true_mart}
For all $\mu<0$, the process $(X^{\lambda_\mu}_t)_{t\geq 0}$ given by
\begin{align}\label{eq:additive-martingale}
    X^{\lambda_\mu}_t := \sum_{\alpha \in I_t} {\tt f}_\mu(M_t^\alpha+\lambda_\mu t )+ \sum_{\beta \in J_t} {\tt g}_\mu(M_t^\beta+\lambda_\mu t), \quad t \ge 0
\end{align}
is a square-integrable nonnegative martingale wrt $(\mathcal F_t)_{t \ge 0}$. 
In particular, $(X^{\lambda_\mu}_t)_{t\ge0}$ converges almost surely to a nonnegative integrable random variable $X^{\lambda_\mu}$ as $t\to\infty$.
\end{prop}

The use of these so-called additive martingales dates back to Watanabe \cite{Wat67}.
This seems to be the earliest reference where martingale methods were used for the study of branching Markov processes, which were systematically introduced around the same time in \cite{INW68a,INW68b, INW69}. In the following, we briefly describe how our model (i.e. on/off BBM) fits into this framework; we temporarily adapt our notation accordingly. 

Let $(X_t)_{t\ge0}=(B_t,\sigma_t)_{t\ge0}$ denote an on/off Brownian motion (\emph{without} branching) with switching rates $c$ resp.\ $c'$ into resp.\ out of dormancy, where $B_t$ denotes the spatial position and $\sigma_t$ the type at time $t\ge0$. This is a strong Markov process with state space $S:=\R\times\{\boldsymbol{a},\boldsymbol{d}\}$, a generic element of which we denote as $(x,\sigma)$.
The corresponding semigroup is Feller on $\mathcal{C}_0(S)$ with infinitesimal generator
\begin{align*}
\mathcal{A}h(x,\sigma)&=\ind_{\{\sigma=\boldsymbol{a} \} }\left(\frac{1}{2}\partial_x^2h(x,\boldsymbol{a})+c\left(h(x,\boldsymbol{d})-h(x,\boldsymbol{a})\right)\right)+
\ind_{\{\sigma=\boldsymbol{d}\}}c'\left(h(x,\boldsymbol{a})-h(x,\boldsymbol{d})\right)
\end{align*}
for each $h$ such that $h(x,\boldsymbol{a})$ is a $\mathcal{C}^2_c$-function in the spatial variable $x$.
From this `single-particle motion', the on/off BBM $(M_t)_{t\ge0}$ with state space $\Gamma$ of \eqref{eq:statespace} can be constructed as a 
branching Markov process in the sense of \cite{INW68a,INW68b, INW69}, where at state-dependent branching rate 
$$k(x,\sigma):=\kappa\cdot \ind_{\{\sigma=\boldsymbol{a}\}}$$
a particle at position $x$ and of type $\sigma$ branches into $l+1$ particles (located at the same position and of the same type) with probability $p_l$. 
We remark that in the terminology of \cite{INW69}, Equation \eqref{eq:FKPPDormancy} is then just the (differential form of the) so-called S-equation for this branching Markov process. 

Now following \cite[Def. 4.10, p. 138]{INW69}, we can define the \emph{expectation semigroup} of $(M_t)_{t\ge0}$ by
\begin{equation}\label{expectation-semigroup}
\mathcal{M}_th(x,\sigma):=\E_{(x,\sigma)}\left[\sum_{\alpha=1}^{N_t} h(M_t^\alpha,\sigma_t^\alpha)\right],\qquad (x,\sigma)\in S, \; t\ge0
\end{equation}
for each bounded or nonnegative function $h:S\to\R$. 
(Note that for an indicator function $h=\ind_D$, $\mathcal{M}_t\ind_D$ gives the expected number of particles of the on/off BBM in the set $D$ at time $t$.)
In our case, the expectation semigroup is a strongly continuous semigroup of bounded operators on $\mathcal{C}_0(S)$ with infinitesimal generator
\begin{equation}\label{generator-expectation-semigroup}
\mathcal{A} h+(\varrho-1)k h,
\end{equation}
where 
$$\varrho:=\sum_{l\in\N}p_l(l+1)$$ denotes the reproduction mean of the offspring distribution, see e.g. \cite[Thm. 4.14, p. 143]{INW69}.
Thus, the expectation semigroup \eqref{expectation-semigroup} is represented in terms of the single-particle motion $X$ as a Feynman-Kac semigroup
\begin{equation}\label{Feynman-Kac}
\mathcal{M}_th(x,\sigma)=\E_{(x,\sigma)}\left[\exp\left(\int_0^t(\varrho-1)k(X_r)\,\ddd r\right)h(X_t)\right],
\end{equation}
see e.g. \cite[p. 210]{Wat67}. In fact, this follows since the RHS of \eqref{expectation-semigroup} and \eqref{Feynman-Kac} are easily seen to have the same generator, namely \eqref{generator-expectation-semigroup}. Note that \eqref{Feynman-Kac} is an instance of what are today commonly called ``many-to-one'' formulae, expressing the expected number of particles of the branching process in terms of the single-particle motion $X$.\footnote{Indeed, \cite{Wat67} is the earliest reference known to us where such a ``many-to-one'' formula was established.}

Now the central observation due to \cite{Wat67} is that each eigenfunction $h$ of the 
expectation semigroup in the sense that 
\[\mathcal{M}_th=e^{-t\lambda}h,\qquad t\ge0\]
induces in a natural way a martingale via
\[Z^\lambda_t:=e^{t\lambda}\sum_{\alpha=1}^{N_t} h(M_t^\alpha,\sigma^\alpha_t)\]
see \cite[p. 216]{Wat67}.
This is in particular satisfied if $h$ is an eigenfunction of the infinitesimal generator \eqref{generator-expectation-semigroup} with eigenvalue $\lambda$.
 In our context, we want to apply the above to
the function $h=h_\mu:S\to\R$ defined by
\begin{equation}\label{eq-eigenfunction}
h_\mu(x,\sigma):=\ind_{\{\sigma=\boldsymbol{a}\}}\,{\tt f_\mu}(x)+\ind_{\{\sigma=\boldsymbol{d}\}}\,{\tt g_\mu}(x),
\end{equation}
with ${\tt f_\mu},{\tt g_\mu}$ from \eqref{eq:f,g-specific}. Recalling the eigenvalue problem \eqref{eq:eigenvalue_problem}, we see that at least formally, this is an eigenfunction of the generator \eqref{generator-expectation-semigroup} with eigenvalue $-\mu\lambda_\mu$.
However, we have to be a little bit careful because since ${\tt f_\mu},{\tt g_\mu}$ are unbounded, the above $h_\mu$ is not in the domain of the infinitesimal generator. 
In fact, it is not even a priori clear that $\mathcal{M}_th_\mu$ is finite for $t>0$. 
However, this needs only a small extra argument in the proof below, for which
we return to our previous notation and rewrite the ``many-to-one'' formula \eqref{Feynman-Kac} as follows: Write $I$ resp.\ $J$ for the random set of time points 
when the on/off Brownian motion $(B_t,\sigma_t)_{t\ge0}$ is active resp.\ dormant; in other words
\[r\in I\quad\Longleftrightarrow\quad\sigma_r=\boldsymbol{a} \qquad\text{and}\qquad r\in J\quad\Longleftrightarrow\quad\sigma_r=\boldsymbol{d}.\]
Then since $\kappa(\varrho-1)= \kappa \sum_{l=1}^\infty p_l l={\tt s}$ (recall \eqref{eq:def-s}), the ``many-to-one'' formula \eqref{Feynman-Kac} takes the form
\begin{align}\label{many-to-one}
    \E_{(x, {\sigma})}\left[\sum_{\alpha \in I_t} {\tt f_\mu}(M_t^\alpha )+ \sum_{\beta \in J_t} {\tt g_\mu}(M_t^\beta)\right]=\E_{(x, \sigma)}\left[e^{\int_0^t {\tt s}\ind_I(r) \, \ddd r} \big({\tt f_\mu}(B_t) \ind_I(t)+ {\tt g_\mu}(B_t) \ind_J(t)\big)\right].
\end{align}

\begin{proof}[Proof of Proposition \ref{prop:true_mart}]

Let $\mu<0$ and define 
\begin{equation}\label{eq-xi}
\xi_t:=\exp\left({\int_0^t {\tt s}\ind_I(r) \, \ddd r}\right) \left({\tt f}_\mu(B_t) \ind_I(t)+ {\tt g}_\mu(B_t) \ind_J(t)\right),\qquad t\ge0
\end{equation}
as the integrand on the RHS of \eqref{many-to-one}. 
Applying It\^{o}'s formula to the process $(\xi_t)_{t\ge0}$ between successive jump (switching) times and compensating the jumps, we see that the process
\begin{align*}
\zeta_t:=\xi_t-\xi_0&- \int_0^t e^{\int_0^r {\tt s}\ind_I(\nu)\, \ddd \nu } \left[\left(\frac{1}{2}  {\tt f}''_\mu (B_r )+ c \left({\tt g}_\mu (B_r )-{\tt f}_\mu (B_r )\right)+{\tt s} {\tt f}_\mu(B_r) \right) \ind_I(r)
    \right.\\
    &\qquad\qquad\qquad\qquad + c' \left({\tt f}_\mu (B_r )-{\tt g}_\mu (B_r )\right)\ind_J(r)  \Bigg ]\,\ddd r,\qquad t\ge0
\end{align*}
is a local martingale null at zero.
But again recalling the eigenvalue problem \eqref{eq:eigenvalue_problem}, since $-\mu \lambda_\mu$ is the Perron-Frobenius eigenvalue of $\frac{1}{2}\mu^2 A +Q +R$ we have that
\begin{align*}
   \zeta_t
    &= \xi_t-\xi_0+ \mu\lambda_\mu\int_0^te^{\int_0^r {\tt s}\ind_I(\nu) \, \ddd \nu} \left({\tt f_\mu}(B_r) \ind_I(r)+ {\tt g_\mu}(B_r) \ind_J(r)\right)\,\ddd r \\
    &= \xi_t-\xi_0+ \mu\lambda_\mu\int_0^t\xi_r\,\ddd r .
    \end{align*}
We will show below that 
\begin{equation}\label{moment-bound-xi}
\E_{(x,\sigma)}\left[\sup_{0\le r\le t}\xi_r\right]
<\infty\end{equation}
for all $t\ge0$ and $(x,\sigma)\in S$.
Together with the local martingale property of $(\zeta_t)_{t\ge0}$, 
this implies by a simple localization argument 
that $(\zeta_t)_{t\ge0}$ is a true martingale and thus
\begin{align*}
   &\E_{(x, \sigma)}\left[\xi_t\right]=\E_{(x, \sigma)}\left[\xi_0\right]- \mu\lambda_\mu\int_0^t \E_{(x, \sigma)}\left[\xi_r\right]\,\ddd r ,\qquad t\ge0.
\end{align*}
Hence we have
\[  \E_{(x, \sigma)}\left[\xi_t\right]=\E_{(x, \sigma)}\left[\xi_0\right]e^{-t\mu\lambda_\mu}, 
\qquad t\ge0,\]
and thus the function $h_\mu$ of \eqref{eq-eigenfunction} is indeed an eigenfunction of the expectation semigroup $(\mathcal{M}_t)_{t\ge0}$ with eigenvalue $e^{-t\mu\lambda_\mu}$.
Now the same calculation as in \cite[ p.\ 216, eq.\ (3.18)-(3.20)]{Wat67}, using the Markov property and the branching property of the on/off BBM, shows the martingale property for the process $(X_t^{\lambda_\mu})_{t\ge0}$ defined in \eqref{eq:additive-martingale}.

 In order to establish \eqref{moment-bound-xi}, we observe that $\xi_t\le e^{t{\tt s}}d_1(\mu)e^{\mu B_t}=:C_{t,\mu}e^{\mu B_t}$
and thus
\[\E_{(x,\sigma)}\left[\sup_{0\le r\le t}\xi_r\right]\le C_{t,\mu}\,\E_{(x,\sigma)}\left[\sup_{0\le r\le t}\exp\left(\mu B_r\right)\right]=C_{t,\mu}e^{\mu x}\,\E_{(0,\sigma)}\left[\exp\left(\mu\inf_{0\le r\le t} B_r\right)\right].\]
On the other hand, we can couple the on/off Brownian motion $(B_t,\sigma_t)_{t\ge0}$ to a \emph{standard} Brownian motion $(\tilde B_t)_{t\ge0}$ such that almost surely
\[\inf_{r\in[0,t]} B_r\ge \inf_{r\in[0,t]}\tilde B_r,\qquad t\ge0,\]
from which the assertion easily follows.

Finally, for the square-integrability of $X_t^{\lambda_\mu}$ 
we observe that by \cite[eq.\ (4.97)]{INW69} (see also p.\ 146) we have
\[    \E_{(x, {\sigma})}\left[\left(\sum_{\alpha \in I_t} {\tt f_\mu}(M_t^\alpha )+ \sum_{\beta \in J_t} {\tt g_\mu}(M_t^\beta)\right)^2\right]= \mathcal{M}_t(h_\mu^2)(x,\sigma)+\int_0^t\mathcal{M}_{t-s}\left(k(\cdot)\mathcal{M}_sh_\mu(\cdot)^2\right)(x,\sigma)\,ds,
\]
which is easily seen to be finite for each $t\ge0$. 
Thus $X_t^{\lambda_\mu}\in L^2$. 

As a nonnegative martingale, $(X_t^{\lambda_\mu})_{t\ge0}$ has an almost sure limit $X^{\lambda_\mu}\in L^1$, whence the proof of Proposition \ref{prop:true_mart} is now finished. 
\end{proof}

\subsection{$L^1$-convergence of $(X_t^{\lambda_\mu})_{t\ge0}$ on $]\mu^*, 0[$}

\noindent In this section, we establish the $L^1$-convergence of the additive martingale $(X_t^{\lambda_\mu})_{t\ge0}$ from \eqref{eq:additive-martingale}
(which we already know to converge almost surely) 
and show that the limit $X^{\lambda_\mu}$ is almost surely positive
whenever $\mu^*<\mu < 0$. 
We will need the following lemma, which is due to \cite[p.\ 229]{N88}.

\begin{lemma}[Neveu]\label{lemma:Nevue}
Let $p\in\ ]1,2]$, $n\in \N$ and $X_1,\ldots, X_n \in L^p$ be a collection of nonnegative independent random variables. Then we have for any $c_1,\dots,c_n \geq 0$ 
\begin{align*}
    \E\left[\left(\sum_{k=1}^n c_kX_k\right)^p \right]- \E\left[\sum_{k=1}^n c_kX_k \right]^p\leq \sum_{k=1}^n c_k^p\left( \E\left[X_k^p\right]-\E\left[X_k\right]^p \right).
\end{align*}
\end{lemma}

\begin{thm} \label{thm:convergence}
Let $\mu^*< \mu < 0$. Then, for any initial condition of the underlying on/off BBM, the corresponding additive martingale $(X^{\lambda_\mu}_t)_{t \geq 0}$ converges to $X^{\lambda_\mu}$ almost surely and in $L^1$ as $t\to\infty$.
\end{thm}

\begin{proof}
The proof follows along the lines of \cite[p. 229]{N88} and \cite[Theorem 1.39]{C97}. Note that by the branching property of the on/off BBM, 
we have a decomposition (see e.g. \cite[eq. (3.2)]{N88}) 
\begin{align} \label{eq:branchingprop}
    X^{\lambda_\mu}_{s+t} = \sum_{\alpha \in I_s} \exp(\mu (M_s^\alpha +\lambda_\mu s))Z^\alpha_t(s) +\sum_{\beta \in J_s} \exp(\mu (M_s^\beta +\lambda_\mu s))Z^\beta_t(s),\qquad  s,t>0,
\end{align}
where conditionally on $\mathcal{F}_s$, the processes $(Z^\alpha_t(s))_{t\ge0}$ for $\alpha \in I_s$ are independent versions of $(X_t^{\lambda_\mu})_{t\ge0}$ started from an active particle at $0$, which are also independent of $\mathcal{F}_s$. The processes $(Z^\beta_t(s))_{t\ge0}$ are defined analogously but started from a dormant particle at $0$. 
 Next, we apply Neveu's Lemma \ref{lemma:Nevue} 
conditionally on $\mathcal{F}_s$ to obtain, for any $p \in\, ]1,2]$, 
\begin{align*}
    \E\left[ (X^{\lambda_\mu}_{t+s})^p\vert \mathcal{F}_s\right] -(X^{\lambda_\mu}_s)^p &\leq  \sum_{\alpha \in I_s} \exp\left(\mu p (M_s^\alpha +\lambda_\mu s)\right)\left(\E\left[(Z_t^\alpha(s))^p\right] - \left(\E\left[Z_t^\alpha(s)\right]\right)^p\right)\\
    &\qquad +\sum_{\beta \in J_s} \exp\left(\mu p (M_s^\beta +\lambda_\mu s)\right)\left(\E\left[(Z_t^\beta(s))^p\right]- \left(\E\left[Z_t^\beta(s)\right]\right)^p\right)\\
    &
    \le C_{t,p,\mu} \sum_{\alpha \in K_s} \exp(\mu p (M_s^\alpha +\lambda_\mu s))
\end{align*}
for some finite constant $C_{t,p,\mu}$, where we use that $X_t^{\lambda_\mu}\in L^2$ by Proposition \ref{prop:true_mart}.
Thus by taking expectations, we get
\begin{align*}
    \E_{(x,\sigma)}\left[ (X^{\lambda_\mu}_{t+s})^p\right] - \E_{(x,\sigma)}\left[(X^{\lambda_\mu}_s)^p\right] &\leq C_{t,p,\mu}\, \E_{(x,\sigma)}\left[\sum_{\alpha \in K_s} \exp(\mu p (M_s^\alpha +\lambda_\mu s))\right]
\end{align*}
for all $(x,\sigma)\in\R\times\{\boldsymbol{a},\boldsymbol{d}\}$.
Since Proposition \ref{prop:speed_function} implies that the mapping $\mu \mapsto \lambda_\mu$ is strictly increasing on $[\mu^*,0[$, we can find some $p\in\, ]1,2]$ such that $\lambda_\mu >\lambda_{\mu p}$. But then 
\begin{align*}
\E_{(x,\sigma)}\left[\sum_{\alpha \in K_s} \exp(\mu p (M_s^\alpha +\lambda_\mu s))\right] &=  \exp\left(\mu p (\lambda_\mu-\lambda_{\mu p})s\right) \,\E_{(x,\sigma)}\left[\sum_{\alpha \in K_s} \exp(\mu p (M_s^\alpha +\lambda_{\mu p} s))\right]\\
    &\leq\exp\left(\mu p (\lambda_\mu-\lambda_{\mu p})s\right) \frac{1}{ d_2(\mu)}\, \E_{(x,\sigma)}\left[X^{\lambda_{\mu p}}_s\right]\\
    &\leq \tilde C_{p,\mu} \exp\left(\mu p (\lambda_\mu-\lambda_{\mu p})s\right),
\end{align*}
since $(X_t^{\lambda_{\mu p}})_{t\ge0}$ is a martingale (also recall \eqref{eq:eigenvector}). This implies that for every $n \in \N$
\begin{align*}
   \E_{(x,\sigma)}\left[(X^{\lambda_\mu}_{n})^p\right]-\E_{(x,\sigma)}\left[(X^{\lambda_\mu}_{0})^p\right] &= \sum_{k=0}^{n-1} \E_{(x,\sigma)}\left[(X^{\lambda_\mu}_{k+1})^p-(X^{\lambda_\mu}_{k})^p\right] \\
   &\leq \widehat C_{1,p,\mu} \sum_{k \in \N} \exp(\mu p (\lambda_\mu-\lambda_{\mu p}) k)< \infty,
\end{align*}
since $\mu< 0$ and $\lambda_\mu>\lambda_{\mu p}$. Hence, $(X^{\lambda_\mu}_n)_{n \in \N}$ is bounded in $L^p$. But since $\big((X^{\lambda_\mu}_t)^p\big)_{t\geq 0}$ is a submartingale, we get $L^p$-boundedness for $(X^{\lambda_\mu}_t)_{t\geq 0}$. The martingale convergence theorem then gives the desired result.
\end{proof}

\noindent The next step is to show that the limiting random variable $X^{\lambda_\mu}$ 
is strictly positive almost surely.

\begin{prop}\label{prop:almost-surely-positive}
Let $\mu^*< \mu < 0$. Then 
we have 
$$\PP_{(x,\sigma)}(X^{\lambda_\mu}>0)=1$$
for all $(x,\sigma)\in\R\times\{\boldsymbol{a},\boldsymbol{d}\}$.
\end{prop}
\begin{proof}
We will see below that 
\begin{equation}\label{eq:0-1}
\PP_{(x,\sigma)}(X^{\lambda_\mu}>0)=1 \quad \mbox{ or } \quad \PP_{(x,\sigma)}(X^{\lambda_\mu}=0)=1.
\end{equation}
Then the assertion follows from the $L^1$-convergence of $(X_t^{\lambda_\mu})_{t\ge0}$ and the fact that $\E_{(x,\sigma)}[X^{\lambda_\mu}_0]>0$.

The proof of \eqref{eq:0-1} uses standard arguments. 
Note that by the branching property we have 
\begin{align}\label{eq:branchingprop2}
    X^{\lambda_\mu}= \sum_{\alpha \in K_s} \exp(\mu(M^\alpha_s+\lambda_\mu s)) Z^{\alpha}(s),\qquad s>0,
\end{align}
where conditionally on $\mathcal{F}_s$, the random variables $Z^{\alpha}(s)$ for $\alpha \in K_s$ are independent copies of 
 $X^{\lambda_\mu}$ and the underlying on/off BBM is started from one particle at $0$ in the state of $M^\alpha_s$, which are moreover independent of $\mathcal{F}_s$.
Indeed, this follows by taking $t\to\infty$ in \eqref{eq:branchingprop}. 

Now let $\sigma\in\{\boldsymbol{a},\boldsymbol{d}\}$ and observe that $X^{\lambda_\mu}$ under $\PP_{(x,\sigma)}$ has the same distribution as $X^{\lambda_\mu}e^{\mu x}$ under $\PP_{(0,\sigma)}$. Hence, the probability
$$\PP_{(x,\sigma)}(X^{\lambda_\mu}=0)=\PP_{(0,\sigma)}(X^{\lambda_\mu}e^{\mu x}=0)=\PP_{(0,\sigma)}(X^{\lambda_\mu}=0)=:p$$
is independent of $x\in\R$, and we omit the subscript for the initial condition of the underlying on/off BBM in the rest of this proof.
By \eqref{eq:branchingprop2}, we have for each $s>0$
\begin{align*}
    p&= \PP(X^{\lambda_\mu} =0)\\
    &= \PP\left(\forall \alpha \in K_s \colon Z^{\alpha}(s)=0 \right)\\
    &= \E\left[ \prod_{\alpha\in K_s}\PP\left( Z^{\alpha}(s)=0\big| \mathcal{F}_s\right)\right]\\
    &= \E[p^{\abs{K_s}}].
\end{align*}

We now claim that $(p^{\abs{K_t}})_{t \geq 0}$ is an $(\mathcal{F}_t)_{t\ge0}$-martingale. Using the Markov and branching properties and our preceding calculation, we infer that, almost surely,
\begin{align*}
    \E\left[ p^{\abs{K_t}} \big\vert \mathcal{F}_s \right] 
    &= \E_{M_s}\left[ p^{\abs{K_{t-s}}} \right]\\
    &= \prod_{\alpha\in K_s} \E\left[p^{\abs{K_{t-s}}}\right]\\
    &= p^{\abs{K_s}}
\end{align*}
for $t>s$.
Now assuming $0<p<1$, we would have $p^{\abs{K_t}}\to0$ (since $|K_t|\to\infty$) almost surely. 
But since $(p^{\abs{K_t}})_{t\geq 0}$ is bounded, the martingale convergence theorem yields also $L^1$-convergence, thus leading to a contradiction.
\end{proof}

\subsection{Almost sure convergence of $(X_t^{\lambda_\mu})_{t\ge0}$ to $0$ on $]-\infty, \mu^*[$}

In this section, we show that the additive martingale  $(X^{\lambda_\mu}_t)_{t \geq 0}$ converges almost surely to zero if $\mu < \mu^\ast < 0$. 
We suitably adapt the reasoning of \cite[Lemma 4.11]{C97}
and begin by examining the diagonal entries of the matrix $\frac{1}{2}\mu^2A+Q+R+\mu \lambda I_2$, 
cf.\ the eigenvalue problem \eqref{eq:eigenvalue_problem}.
For $\mu,\lambda\in\R$, define
 \[F_{\boldsymbol{a}}(\mu,\lambda):=\tfrac{1}{2}\mu^2-c+{\tt s}+\mu\lambda\qquad\text{and}\qquad F_{\boldsymbol{d}}(\mu,\lambda):=\mu\lambda-c',\]
so that we have 
 $$\tfrac{1}{2}\mu^2A+Q+R+\mu \lambda I_2 = \begin{pmatrix}F_{\boldsymbol{a}}(\mu,\lambda) & c \\ c' & F_{\boldsymbol{d}}(\mu,\lambda)\end{pmatrix}.$$

\begin{lemma}[Negativity Lemma]\label{negativity-lemma}
For each fixed $\mu<0$ and $\sigma \in \{\boldsymbol{a},\boldsymbol{d}\}$, we have $F_\sigma(\cdot,\lambda_\mu)<0$ in a suitable neighborhood of $\mu$. 
\end{lemma}
\begin{proof}
Let $\mu<0$. By Lemma \ref{lem:speedfunctionpositivity}, there exists a strictly positive eigenvector $\vec{d}(\mu)$ of $\frac{1}{2}\mu^2A+Q+R$ with corresponding eigenvalue $-\mu \lambda_\mu$.
    The assertion then follows from $\left(\frac{1}{2}\mu^2A+Q+R+\mu \lambda _\mu I_2\right) \vec{d}(\mu) = \vec 0$ together with the model assumption $c, c' > 0$ and the continuity of $F_\sigma(\cdot,\lambda_\mu)$.
\end{proof}

\noindent Next, we provide an upper bound on the expectation of the limiting random variable $X^{\lambda_\mu}$ 
of the additive martingale from Proposition \ref{prop:true_mart}. For $p\in\ ]0,1[$, define $\vec{\theta}(p):=( \theta_{\boldsymbol{a}}(p), \theta_{\boldsymbol{d}}(p))^T$ with
    \begin{align*}
        \theta_{\boldsymbol{a}}(p) := \mathbb{E}_{(0,\boldsymbol{a})}[(X^{\lambda_\mu})^p] \qquad\text{and}\qquad
        \theta_{\boldsymbol{d}}(p) := \mathbb{E}_{(0,\boldsymbol{d})}[(X^{\lambda_\mu})^p] .
    \end{align*}
Further, recall that we denote by $\kappa$ the branching rate of the on/off BBM, by $\varrho=\sum_{k=2}^\infty kp_{k-1}$ the reproduction mean of the corresponding offspring distribution, and that we have ${\tt s}=\kappa(\varrho-1)$, see \eqref{eq:def-s}.

\begin{lemma}\label{lemma-expectation}
    Let $\mu < \mu^\ast < 0$. Then for all $p \in\, ]0,1[$ close enough to 1 we have
    \begin{align*}
        \theta_{\boldsymbol{a}}(p)         \leq \frac{c\theta_{\boldsymbol{d}}(p)+(\kappa+{\tt s})\theta_{\boldsymbol{a}}(p)}{-F_{\boldsymbol{a}}(\mu p,\lambda_\mu)+\kappa+{\tt s}}
\qquad\text{and}\qquad
        \theta_{\boldsymbol{d}}(p)        \leq \frac{c'\theta_{\boldsymbol{a}}(p)}{-F_{\boldsymbol{d}}(\mu p,\lambda_\mu)}.
    \end{align*}
    \end{lemma}
    \begin{proof}
Let $T_1$ denote the first switching time and $T_2$ the first branching time of the on/off BBM, started from a single (active or dormant) particle. 
Then by the branching property, we have a decomposition 
\begin{align*}
    X^{\lambda_\mu} &=  X^{\lambda_\mu} \ind_{\lbrace T_1 < T_2 \rbrace} + X^{\lambda_\mu} \ind_{\lbrace T_2 < T_1 \rbrace}\sum_{k=2}^\infty\ind_{\lbrace |K_{T_2}|=k  \rbrace}\\
    &= e^{\mu(M^1_{T_1}+\lambda_\mu T_1)} Z^{\mu,0}\ind_{\lbrace T_1 < T_2 \rbrace} + \ind_{\lbrace T_2 < T_1 \rbrace}\sum_{k=2}^\infty\ind_{\lbrace |K_{T_2}|=k \rbrace}\sum_{\alpha \in K_{T_2}} e^{\mu(M^1_{T_2}+\lambda_\mu T_2)} Z^{\mu,\alpha}\qquad \text{a.s.},
\end{align*}
where the random variables $Z^{\mu, \alpha}$ for $\alpha \in K_{T_2}$ are (conditionally on $\mathcal{F}_{T_2}$) independent copies of $X^{\lambda_\mu}$ 
with the underlying on/off BBM started from one particle at $0$ in the state of $M^\alpha_{T_2}$, which are moreover independent of $\mathcal{F}_{T_2}$. Similarly,  $Z^{\mu,0}$ is a copy of $X^{\lambda_\mu}$ started from one particle at $0$ in the state of $M^1_{T_1}$, independent of $\mathcal{F}_{T_1}$.
Indeed, this follows by taking $s=T_1$ resp.\ $s=T_2$ in \eqref{eq:branchingprop2}.
From this decomposition we derive, using that $(z_1+z_2)^p\le z_1^p+z_2^p$ for all $z_1,z_2\ge0$ since $0<p<1$, that 
        \begin{align*}
            \mathbb{E}_{(0,\boldsymbol{a})}\left[(X^{\lambda_\mu})^p\right] &\leq \mathbb{E}_{(0,\boldsymbol{a})}\left[e^{\mu p(M^1_{T_1} + \lambda_\mu T_1)}\ind_{\lbrace T_1 < T_2\rbrace } \theta_{\boldsymbol{d}}(p)\right] + \varrho\,\mathbb{E}_{(0,\boldsymbol{a})}\left[e^{\mu p(M^1_{T_2} + \lambda_\mu T_2)}\ind_{\lbrace T_2<T_1\rbrace}\theta_{\boldsymbol{a}}(p)\right].
        \end{align*}
Since (starting from a single active particle) $T_1$ resp.\ $T_2$ are independent exponential random variables with parameters $c$ resp.\ $\kappa$, evaluating the above expectations gives
        \begin{align*}
\theta_{\boldsymbol{a}}(p) &\leq -\frac{c\theta_{\boldsymbol{d}}(p)+\kappa\varrho \theta_{\boldsymbol{a}}(p)}{\frac{1}{2}(\mu p)^2+\mu p\lambda_\mu -\kappa -c} = \frac{c\theta_{\boldsymbol{d}}(p)+ (\kappa+{\tt s}) \theta_{\boldsymbol{a}}(p)}{-F_a(\mu p,\lambda_\mu)+\kappa+{\tt s}}.
        \end{align*}
Starting from a dormant particle at $0$, the argument is analogous but simpler since then the initial particle 
first has to wake up before it can branch or move. In particular, we have $T_1 < T_2$ and $M^1_{T_1}=0 $ almost surely and therefore get 
        \begin{equation*}
\mathbb{E}_{(0,\boldsymbol{d})}\left[(X^{\lambda_\mu})^p\right] 
= \mathbb{E}_{(0,\boldsymbol{d})}\left[e^{\mu p \lambda_\mu T_1}\theta_{\boldsymbol{a}}(p)\right] = \frac{c'}{c' - \mu p \lambda_\mu}\theta_{\boldsymbol{a}}(p) = \frac{c' \theta_{\boldsymbol{a}}(p)}{-F_{\boldsymbol{d}}(\mu p,\lambda_\mu)}.
        \end{equation*}
    \end{proof}

\begin{prop}
	\label{prop:zero}
    For $\mu < \mu^\ast < 0$ and $p \in\ ]0,1[$  
    close enough to $1$ it holds
    \begin{align*}
        \mathbb{E}_{(0,\boldsymbol{a})}\left[(X^{\lambda_\mu})^p\right] = 0 \quad \mbox{ and } \quad  \mathbb{E}_{(0,\boldsymbol{d})}\left[(X^{\lambda_\mu})^p\right] = 0. 
    \end{align*}
    In particular $X^{\lambda_\mu} = 0$ almost surely.
\end{prop}

\begin{proof}
    Since the additive martingale is nonnegative and converges almost surely, we have 
$ \theta_{\boldsymbol{a}}(p), \theta_{\boldsymbol{d}}(p)\ge0$.

 Due to the Negativity Lemma \ref{negativity-lemma}, we have 
$-F_{\boldsymbol{a}}(p\mu,\lambda_\mu)+\kappa+{\tt s} > 0$ and $-F_{\boldsymbol{d}}(p\mu,\lambda_\mu) > 0$
  for $p$ close enough to 1. 
Thus we can rewrite the inequalities from Lemma \ref{lemma-expectation} to obtain
\begin{equation}\label{estimate-nonnegative}
    \begin{pmatrix}
        0\\0
    \end{pmatrix} \leq \begin{pmatrix}
        F_{\boldsymbol{a}}(p\mu,\lambda_\mu) & c \\
        c' & F_{\boldsymbol{d}}(p\mu,\lambda_\mu)
    \end{pmatrix} \begin{pmatrix}
        \theta_{\boldsymbol{a}}(p) \\ \theta_{\boldsymbol{d}}(p)
    \end{pmatrix} = \left(\tfrac{1}{2}(\mu p)^2A+Q+R+\mu p \lambda_\mu I_2 \right)\vec{\theta}(p).
\end{equation}

\noindent 
Inverting the matrix $\frac{1}{2}(\mu p)^2A+Q+R+\mu p \lambda_\mu I_2$, 
we get

\begin{equation*}
    \left(\tfrac{1}{2}(\mu p)^2A+Q+R+\mu p \lambda_\mu I_2\right)^{-1} = \frac{1}{
    \det\left(\tfrac{1}{2}(\mu p)^2A+Q+R+\mu p \lambda_\mu I_2\right)}
    \begin{pmatrix}
        F_{\boldsymbol{d}}(\mu p,\lambda_\mu) & -c \\
        -c' & F_{\boldsymbol{a}}(\mu p,\lambda_\mu)
    \end{pmatrix}.
\end{equation*}
Now write $P(\mu p,\lambda_\mu)$ for the determinant in the above expression. Then $P(\cdot,\lambda_\mu)$ is a polynomial of degree $3$ with a positive leading coefficient and three distinct zeroes, the smallest of which is $\mu$.\footnote{Since $\mu^*$ is the only local minimum of the speed function and $\mu<\mu^*$, we have that $\mu_1:=\mu$ must be the smallest zero of $P(\cdot,\lambda_\mu)$. Since $\lim_{\nu\to0-}\lambda_\nu=\infty$ (see \eqref{eq:asymptotics-speed-function}), there is exactly one other value $\mu^*<\mu_2<0$ with $\lambda_{\mu_2}=\lambda_\mu$, corresponding to the second zero. The third zero is obtained as the unique $\mu_3>0$ with $\lambda^-_{\mu_3}=\lambda_\mu$. This value must exist since $\lim_{\nu \to 0+} \lambda^-_\nu =\infty$ and $\lim_{\nu \to \infty} \lambda^-_\nu =0$.
}
Consequently, if $p \in\, ]0,1[$ is close enough to $1$, then $P(\mu p,\lambda_\mu)$ is positive and the above inverse is a matrix with strictly negative entries.
Together with \eqref{estimate-nonnegative}, this implies 
 $$\vec{\theta}(p)=\left(\tfrac{1}{2}(\mu p)^2A+Q+R+\mu p \lambda_\mu I_2\right)^{-1}\left(\tfrac{1}{2}(\mu p)^2A+Q+R+\mu p \lambda_\mu I_2\right)\vec{\theta}(p) \leq \vec{0}$$ 
and thus $\vec{\theta}(p)=\vec{0}$.
\end{proof}

\subsection{The speed of the rightmost particle of on/off BBM } 
\label{subsec:Speed}

In this section, we derive the asymptotic speed of the rightmost particle of an on/off branching Brownian motion, thereby providing a proof of Theorem \ref{thm:rightmost} (for model variant I).
We define the position of the rightmost and the leftmost particle in the on/off BBM $M_t$  at time $t\ge0$ 
by 
$$
R_t:=\max_{\alpha \in K_t} M_t^\alpha \quad \mbox{ and } \quad L_t:=\min_{\alpha \in K_t} M_t^\alpha.$$
Of course, if we start from a single active resp.\ dormant particle at the origin, i.e.\ $M_0=(0, \boldsymbol{a})$ resp.\ $M_0=(0, \boldsymbol{d})$, then by symmetry $R_t$ and $-L_t$ are equal in law.

We first provide an upper bound on the asymptotic speed of propagation of $R_t$ based on the extinction of the additive martingale $(X_t^{\lambda_\mu})_{t\ge0}$ in the regime $\mu < \mu^\ast < 0$.

\begin{prop}\label{prop:upperbound}
For all $(x,\sigma)\in \R\times\{\boldsymbol{a},\boldsymbol{d}\}$, we have 
\begin{equation*}
    \limsup_{t \to \infty} \frac{R_t}{t} \leq \lambda^* \quad \PP_{(x,\sigma)}\text{-a.s.}
\end{equation*}    
\end{prop}

\begin{proof}
Since for all $x\in\R$
$$ \PP_{(x,\sigma)}\left(\limsup_{t \to \infty} \frac{R_t}{t} \leq \lambda^* \right)= \PP_{(0,\sigma)}\left(\limsup_{t \to \infty} \frac{R_t+x}{t} \leq \lambda^* \right)= \PP_{(0,\sigma)}\left(\limsup_{t \to \infty} \frac{R_t}{t} \leq \lambda^* \right),$$
it clearly suffices to consider $x=0$.
Given $\epsilon > 0$, let $\lambda := \lambda^* + \epsilon$ and
choose $\mu < \mu^\ast < 0$ such that the speed function $\lambda_{\bullet}$ takes the value $\lambda$ at $\mu$, i.e.\ $\lambda=\lambda_\mu$. 
By Propositions \ref{prop:true_mart} and \ref{prop:zero} we have $X_t^{\lambda_\mu} \to 0$ almost surely as $t\to\infty$. 
    Further, by \eqref{eq:f,g-specific} and the definition of $(X_t^{\lambda_\mu})_{t\ge0}$ in \eqref{eq:additive-martingale} we get
    \begin{equation*}
        X_t^{\lambda_\mu} \geq 
        d_2(\mu)\,e^{\mu(L_t + \lambda_\mu t)} > 0 \quad \text{a.s.},
    \end{equation*}
where we also recall \eqref{eq:eigenvector}. 
Since $\mu < 0$, 
 combining both observations implies that  
    $L_t + \lambda_\mu t \to +\infty$ almost surely as $t \to \infty$. In particular, almost surely there is an $N \geq 0$ such that for all $t > N$ we have $L_t + \lambda_\mu t > 0$. From this, by symmetry, we infer that 
    \begin{equation*}
        \lambda^* + \epsilon = \lambda_\mu \geq \limsup_{t \to \infty} \frac{R_t}{t} \quad \text{a.s.}
    \end{equation*}
    Since $\epsilon > 0$ was arbitrary, the assertion follows.
\end{proof}

\noindent Next, we use the almost sure and $L^1$-convergence of $(X_t^{\lambda_\mu})_{t\ge0}$ to a positive random variable  in the survival  regime $\mu^\ast < \mu < 0$ in order to obtain the corresponding lower bound.

\begin{prop}
	\label{prop:lowerbound}
For all $(x,\sigma)\in \R\times\{\boldsymbol{a},\boldsymbol{d}\}$, we have
\begin{align*}
    \liminf_{t \to \infty} \frac{R_t}{t} \geq \lambda^* \quad\PP_{(x,\sigma)}\text{-a.s.}
\end{align*}
\end{prop}

\begin{proof}
This follows along the lines of \cite[Proof of Theorem 1.44]{C97}. 
Again it suffices to consider $x=0$. Moreover, we suppose that $\sigma=\boldsymbol{a}$; the argument for $\sigma=\boldsymbol{d}$ is analogous.
Fix $\mu^*<\mu<0$. Then we know
by Theorem \ref{thm:convergence} and Proposition \ref{prop:almost-surely-positive} that $X^{\lambda_\mu}$ is an almost surely positive random variable with expectation $\E_{(0,\boldsymbol{a})}[X^{\lambda_\mu}]=\E_{(0,\boldsymbol{a})}[X_0^{\lambda_\mu}]=1$ (again recall \eqref{eq:eigenvector}). Thus,
  we may define an equivalent probability measure by 
 $$
 \ddd \mathbb{Q}^\mu := X^{\lambda_\mu} \ddd \PP_{(0,\boldsymbol{a})}
 $$
and note that by the $L^1$-convergence $X_t^{\lambda_\mu}\to X^{\lambda_\mu}$, we have $ \mathbb{Q}^\mu\big|_{{\mathcal F_t}}=X_t^{\lambda_\mu} \ddd \PP_{(0,\boldsymbol{a})}$ for each $t\ge0$.
By the definition \eqref{eq:additive-martingale} of the additive martingale, \eqref{eq:f,g-specific} and \eqref{eq:lambda_mu} together with the explicit form of the eigenvector from Lemma \ref{lem:speedfunctionpositivity}, we see that for each $t \ge 0$, $X^{\lambda_\mu}_t$ is infinitely often differentiable in $\mu$. 
 Moreover, by interchanging expectation and differentiation we observe that 
first and second partial derivatives with respect to $\mu$ of $(X^{\lambda_\mu}_t)_{t \geq 0}$ are still $\PP_{(0,\boldsymbol{a})}$-martingales.
 Hence, the processes
 \begin{align*}
     N_t^{(1)} \defeq \frac{1}{X_t^{\lambda_\mu}}  {\partial_\mu X^{\lambda_\mu}_t}  \quad \mbox{ and } \quad
     N_t^{(2)} \defeq \frac{1}{X_t^{\lambda_\mu}}  {\partial^2_\mu X^{\lambda_\mu}_t},\qquad t\ge0  
 \end{align*}
 define $\mathbb{Q}^\mu$-martingales.
  Next, we define for $\alpha \in I_t$ and  $\beta \in J_t$ the strictly positive quantities 
 \begin{align*}
     S^\alpha_t:=\frac{{\tt f}_\mu(M^\alpha_t+\lambda_\mu t)}{X^{\lambda_\mu}_t} \quad \mbox{ and } \quad  S^\beta_t:=\frac{{\tt g}_\mu(M^\alpha_t+\lambda_\mu t)}{X^{\lambda_\mu}_t}
 \end{align*}
 that can be regarded as probability weights since 
 \begin{align}
 	\label{eq:prob_weights}
	\sum_{\gamma \in K_t} S_t^\gamma =\sum_{\alpha \in I_t} S_t^\alpha + \sum_{\beta \in J_t} S^\beta_t= \frac{X_t^{\lambda_\mu}}{X_t^{\lambda_\mu}} =1.
\end{align}
A simple computation, taking partial derivatives with respect to $\mu$, 
yields that $N_t^{(1)}$ can now be expressed as 
 \begin{align*} 
	N_t^{(1)}&= \sum_{\alpha \in I_t } S_t^\alpha \left(\frac{\partial_\mu d_1(\mu)}{ d_1(\mu)} +M_t^\alpha +\partial_\mu (\mu \lambda_\mu) t \right)+\sum_{\beta \in J_t } S_t^\beta \left(\frac{\partial_\mu d_2(\mu)}{ d_2(\mu)} +M_t^\beta +\partial_\mu (\mu \lambda_\mu) t \right).
\end{align*}
This leads to the estimates
 \begin{align} \label{eq:N/t}	
	\frac{N_t^{(1)}}{t}&\geq   \min_{i\in \lbrace 1,2\rbrace}\frac{\partial_\mu d_i(\mu)}{ t d_i(\mu)} +\frac{L_t}{t} +\partial_\mu (\mu \lambda_\mu) 
\end{align}	
and
 \begin{align}\label{eq:upper-bound}
	\left(N_t^{(1)}\right)^2&\leq \sum_{\alpha \in I_t } S_t^\alpha \left(\frac{\partial_\mu d_1(\mu)}{ d_1(\mu)} +M_t^\alpha +\partial_\mu (\mu \lambda_\mu) t \right)^2 +\sum_{\beta \in J_t } S_t^\beta \left(\frac{\partial_\mu d_2(\mu)}{ d_2(\mu)} +M_t^\beta +\partial_\mu (\mu \lambda_\mu) t \right)^2,
\end{align}
where we used \eqref{eq:prob_weights} and Jensen's inequality for the upper bound. 
Moreover, another simple but slightly more tedious computation yields that
 \begin{align*} 
	N_t^{(2)}&= \sum_{\alpha \in I_t } S_t^\alpha\left[ \left(\frac{\partial_\mu d_1(\mu)}{ d_1(\mu)} +M_t^\alpha +\partial_\mu (\mu \lambda_\mu) t \right)^2 + \frac{\partial^2_\mu d_1(\mu)}{ d_1(\mu)} - \left(\frac{\partial_\mu d_1(\mu)}{ d_1(\mu)}\right)^2 + \partial^2_\mu (\mu \lambda_\mu) t\right]\\
	&\quad + \sum_{\beta \in J_t } S_t^\beta\left[ \left(\frac{\partial_\mu d_2(\mu)}{ d_2(\mu)} +M_t^\beta +\partial_\mu (\mu \lambda_\mu) t \right)^2 + \frac{\partial^2_\mu d_2(\mu)}{ d_2(\mu)} - \left(\frac{\partial_\mu d_2(\mu)}{ d_2(\mu)}\right)^2 + \partial^2_\mu (\mu \lambda_\mu) t\right].
\end{align*}
Hence, \eqref{eq:upper-bound} now implies
\begin{align*}
    \E_{\mathbb{Q}^\mu}\left[\left(N_t^{(1)}\right)^2\right] &\leq \E_{\mathbb{Q}^\mu}\left[\sum_{\alpha \in I_t } S_t^\alpha \left(\frac{\partial_\mu d_1(\mu)}{ d_1(\mu)} +M_t^\alpha +\partial_\mu (\mu \lambda_\mu) t \right)^2 +\sum_{\beta \in J_t } S_t^\beta \left(\frac{\partial_\mu d_2(\mu)}{ d_2(\mu)} +M_t^\beta +\partial_\mu (\mu \lambda_\mu) t \right)^2 \right]\\
    &=\E_{\mathbb{Q}^\mu}\left[N_t^{(2)}\right] + \E_{\mathbb{Q}^\mu}\left[\sum_{\alpha \in I_t } S_t^\alpha \left( \left(\frac{\partial_\mu d_1(\mu) }{d_1(\mu)}\right)^2 - \frac{\partial^2_\mu d_1(\mu)}{d_1(\mu)} -\partial_\mu^2 (\mu \lambda_\mu)t \right)\right]\\
    &\qquad + \E_{\mathbb{Q}^\mu}\left[\sum_{\beta \in J_t } S_t^\beta \left( \left(\frac{\partial_\mu d_2(\mu) }{d_2(\mu)}\right)^2 - \frac{\partial^2_\mu d_2(\mu)}{d_2(\mu)} -\partial_\mu^2 (\mu \lambda_\mu)t \right)\right] \\
    &\leq C_1^\mu +C_2^\mu t
\end{align*}
for constants 
\begin{align*}
    C_1^\mu &:=\E_{\mathbb{Q}^\mu}[N_0^{(2)}] + \max_{i\in \lbrace 1,2\rbrace}\left (\left(\frac{\partial_\mu d_i(\mu) }{d_i(\mu)}\right)^2 - \frac{\partial^2_\mu d_i(\mu)}{d_i(\mu)}\right)\geq 0
\end{align*}
and 
\begin{align*}
	C_2^\mu &:= -\partial_\mu^2 (\mu \lambda_\mu) \ge 0,
\end{align*}
where we used again \eqref{eq:prob_weights} and the fact that $\left(N_t^{(2)}\right)_{t\ge0} $ is a martingale under $\mathbb{Q}^\mu$. But then Doob's maximal inequality yields for every $\eps>0$ and $n\in \N$
\begin{align*}
    \mathbb{Q}^\mu \left( \sup_{ t\in[2^{n-1}, 2^n]} \tfrac{\abs{N_t^{(1)}}}{t} \geq \eps\right)&\leq \mathbb{Q}^\mu \left( \sup_{ t\in[0, 2^n]} \abs{N_t^{(1)}} \geq \eps2^{n-1}\right) \leq \eps^{-2} 2 ^{-2(n-1)}\big(C_1^\mu+2^n C_2^\mu \big).
\end{align*}
Since the right hand side is summable in $n\in \N$, we get from the Borel-Cantelli Lemma that $\nicefrac{N_t^{(1)}}{t} \xrightarrow{t\to\infty} 0$ almost surely under $\mathbb{Q}^\mu$ und thus also under $\PP_{(0,\boldsymbol{a})}$. 
Now, the estimate \eqref{eq:N/t} implies 
\begin{align*}
    \limsup_{t \to \infty}\frac{L_t}{t} \leq -\partial_\mu(\mu \lambda_\mu)\quad \text{a.s.},
\end{align*}
which holds for every $\mu \in\ ]\mu^*,0[$. Since the 
right hand side is continuous in $\mu$, we may infer that
\begin{align*}
    \limsup_{t \to \infty}\frac{L_t}{t} \leq -\partial_\mu(\mu \lambda_{\mu})\big|_{\mu =\mu^*} \quad \text{a.s.}
\end{align*}
Now, recall that $\mu^*$ is the unique minimizer of the speed function 
 $\mu \mapsto \lambda_\mu$ on the negative half axis, which yields that
$$\partial_{\mu}( \mu\lambda_{\mu}) \big|_{\mu=\mu^ *}=\lambda^* +\mu\partial_\mu( \lambda_{\mu})\big|_{\mu =\mu^*}=\lambda^*$$
and thus
$$\limsup_{t \to \infty}\frac{L_t}{t} \leq -\lambda^*\quad \text{a.s.}$$
Since $L_t$ is equal in law to $-R_t$, the proof is finished.
\end{proof}

Note that taken together, Propositions \ref{prop:upperbound} and \ref{prop:lowerbound} provide a proof of Theorem \ref{thm:rightmost} for model variant I.
In particular, we have proved that for all $(x,\sigma)\in \R\times\{\boldsymbol{a},\boldsymbol{d}\}$
 \begin{align}
 	\label{eq:Rt}
	\lim_{t \to \infty} \frac{R_t}{t}=\lambda^*\quad\text{and}\quad 	\lim_{t \to \infty} \frac{L_t}{t}=-\lambda^*\qquad \PP_{(x,\sigma)}\text{-a.s.}
\end{align}
This already allows us to infer that our PDE \eqref{eq:FKPPDormancy}, when started from Heaviside initial conditions, 
exhibits an (approximate) travelling wave solution with critical speed $\lambda^*$. 

\begin{proof}[Proof of Theorem \ref{thm:critical_speed}]
For Heaviside initial conditions $u_0=v_0=\ind_{\R^+ }$, 
Proposition \ref{prop:duality} implies that
 \begin{align*}
     u(t,x+\lambda t)=\PP_{(0,\boldsymbol{a})}(R_t\leq x +\lambda t) \quad \mbox{ and } \quad 
     v(t,x+\lambda t)=\PP_{(0,\boldsymbol{d})}(R_t\leq x +\lambda t),
 \end{align*}
from which the result follows in combination with \eqref{eq:Rt}. 
\end{proof}

\section{Multiplicative martingales and travelling wave solutions}
\label{sec:travelling}

\noindent The convergence results in the preceding section contain all the necessary tools for the analysis of travelling wave solutions to our original PDE \eqref{eq:FKPPDormancy}. As before, we focus on model variant I, but emphasize that analogous results can be proved for variant II. 
Again, we will employ martingale methods, this time based on so-called multiplicative instead of additive martingales, 
the use of which was initiated by Neveu \cite{N88}.

\begin{prop} \label{prop: multmartingale}
Let $M=(M_t)_{t \geq 0}$ be an on/off BBM as defined in Definition \ref{def:VariantI} and $(\mathcal{F}_t)_{t \ge 0}$ its canonical filtration. 
Suppose $f,g\in\mathcal{C}^2(\R,[0,1])$.
Then $(f, g)$ is a solution to Equation \eqref{eq:TravellingWaveI} iff the stochastic process $(Y^\lambda)_{t\ge 0}$ defined by
\begin{equation*}
Y^\lambda_t:=\prod_{\alpha \in I_t}  f(M_t^\alpha+\lambda t)\prod_{\beta \in J_t}  g(M_t^\beta+\lambda t), \qquad t\ge 0, 
\end{equation*}
is a martingale wrt to $(\mathcal{F}_t)_{t \ge 0}$.
\end{prop}

\begin{proof}
Applying It\^{o}'s formula to the process $(M_t)_{t\ge0}$ between successive jump (i.e.\ switching or branching) times, and compensating the jumps, we see that
\begin{align*}
Y^\lambda_t&=Y^\lambda_0+\text{loc. mart.}\\
&+\int_0^t\Bigg[ \sum_{\alpha \in I_r } \prod_{\gamma \in I_r\setminus \lbrace\alpha  \rbrace}  f(M_r^\gamma +\lambda r ) \prod_{\beta \in J_r}  g(M_r^\beta +\lambda r )\bigg( \frac{1}{2}  f''(M_r^\alpha +\lambda r )+\lambda f'(M_r^\alpha +\lambda r) \\
    &\qquad\qquad +\kappa\sum_{k=1}^\infty p_k \left( f(M_r^\alpha +\lambda r)^{k+1}- f(M_r^\alpha +\lambda r)\right) +c\left( g(M_r^\alpha+\lambda r)- f(M_r^\alpha +\lambda r)\right)\bigg )\\
    &+ \sum_{\beta \in J_r} \prod_{\alpha \in I_r }  f(M_r^\alpha+\lambda r) \prod_{\gamma \in J_r\setminus \lbrace \beta \rbrace}  g(M_r^\gamma +\lambda r) \left( \lambda g'(M_r^\beta +\lambda r)+c' \left(f(M_r^\beta+\lambda r)- g(M_r^\beta+\lambda r)\right)\right)\Bigg] \ddd r .
\end{align*}
Thus if $(f, g)$ solves Equation \eqref{eq:TravellingWaveI}, the last term on the RHS equals zero, hence $(Y^\lambda)_{t\ge 0}$ itself is a local martingale with values in $[0,1]$ and thus a true martingale.

Conversely, if $(Y^\lambda)_{t\ge 0}$ is a martingale, we have
\[f(x)=\E_{(x,\boldsymbol{a})}\left[Y^\lambda_0\right]=\E_{(x,\boldsymbol{a})}\left[Y^\lambda_t\right]\qquad\text{and}\qquad g(x)=\E_{(x,\boldsymbol{d})}\left[Y^\lambda_0\right]=\E_{(x,\boldsymbol{d})}\left[Y^\lambda_t\right]\]
for all $t\ge0$ and $x\in\R$. On the other hand, we know by Proposition \ref{prop:duality} that the unique solution to Equation \eqref{eq:FKPPDormancy} with initial condition $(f,g)$ is given by
\begin{align*}
    u(t,x) &=\E_{(x,\boldsymbol{a})}\left[\prod _{\alpha \in I_t}f(M^\alpha_t)\prod _{\beta \in J_t}g(M^\beta_t)\right], 
	& v(t,x) =\E_{(x,\boldsymbol{d})}\left[\prod _{\alpha \in I_t}f(M^\alpha_t)\prod _{\beta \in J_t}g(M^\beta_t)\right].
\end{align*}
This gives
\[u(t,x)=f(x-\lambda t),\qquad v(t,x)=g(x-\lambda t),\]
i.e.\ $(u,v)$ is a travelling wave solution to Equation \eqref{eq:FKPPDormancy} and $(f,g)$ is a solution to the travelling wave equation \eqref{eq:TravellingWaveI}.
\end{proof}

\subsection{The supercritical case $\lambda > \lambda^*$}
In this section, we establish the existence and properties of solutions to the travelling wave equation \eqref{eq:TravellingWaveI} in the `supercritical case' $\lambda > \lambda^*$, thereby providing in particular a proof for Theorem \ref{thm:existence_of_supercritical_waves} (for model variant I). 
In order to do so, we will use Proposition \ref{prop: multmartingale} and employ the limit $X^{\lambda_\mu}$ of the additive martingale, 
which we recall is almost surely positive for $\mu^*<\mu< 0$,  
to construct a suitable multiplicative martingale.
Let
\begin{align}\label{def-wave-function}
    f_\mu(x) :=\E_{(x,\boldsymbol{a})}\left[e^{-X^{\lambda_\mu}}\right] \quad \mbox{ and } \quad 
     g_\mu(x) :=\E_{(x,\boldsymbol{d})}\left[e^{-X^{\lambda_\mu}}\right],\qquad x\in\R.
\end{align}
Again using that $X^{\lambda_\mu}$ under $\PP_{(x,\sigma)}$ has the same distribution as $X^{\lambda_\mu}e^{\mu x}$ under $\PP_{(0,\sigma)}$ for $\sigma\in\{\boldsymbol{a},\boldsymbol{d}\}$, we can rewrite this definition as
 \begin{align}\label{def-wave-function-2}
     f_\mu(x)= \E_{(0,\boldsymbol{a})}\left[\exp(-X^{\lambda_\mu} e^{\mu x})\right] \quad \mbox{ and } \quad 
     g_\mu(x)= \E_{(0,\boldsymbol{d})}\left[\exp(-X^{\lambda_\mu} e^{\mu x})\right],
 \end{align}
an observation we will use repeatedly. In particular, this shows that $f_\mu,g_\mu\in\mathcal{C}^2(\R,[0,1])$. 

\begin{proof}[Proof of  Theorem \ref{thm:existence_of_supercritical_waves}]
We are guided by 
 \cite[Section 3.2]{H99}. 
 Let $\mu^*< \mu< 0$. For $(f_\mu, g_\mu)$ as defined above, 
we consider
$$
Y^{\lambda_\mu}_t:=\prod_{\alpha \in I_t}  f_\mu(M_t^\alpha+\lambda_\mu t)\prod_{\beta \in J_t}  g_\mu(M_t^\beta+\lambda_\mu t),\qquad t\ge 0. 
$$ 
We will show below that we have the representation 
\begin{equation}\label{eq:representation-martingale}
Y_t^{\lambda_\mu}=\E\left[e^{-X^{\lambda_\mu}}\Big\vert \mathcal{F}_t\right]\quad\text{a.s.}
\end{equation}
for all $t \geq 0$. 
In particular, 
this shows that is a $(Y^{\lambda_\mu}_t)_{t\ge0}$ 
martingale. Thus by Proposition \ref{prop: multmartingale}, $(f_\mu,g_\mu)$ solves the travelling wave equation \eqref{eq:TravellingWaveI} with speed $\lambda=\lambda_\mu$.
Since $X^{\lambda_\mu}>0$ almost surely, Equation \eqref{def-wave-function-2}
immediately yields that both $f_\mu$ and $g_\mu$ are increasing from $0$ to $1$.

In order to establish \eqref{eq:representation-martingale}, first note that by the branching property we have
 the following representation for the limiting random variable $X^{\lambda_\mu}$, see also \eqref{eq:branchingprop2}: For each $t \ge 0$,
 \begin{align} \label{eq:branchingproplimit}
     X^{\lambda_\mu}=\sum_{\alpha \in I_t}Z^\alpha(t) \exp(\mu (M_t^\alpha+\lambda_\mu t))+\sum_{\beta \in J_t}Z^\beta(t) \exp(\mu (M_t^\beta+\lambda_\mu t)),
 \end{align}
 where the random variables $Z^\alpha(t)$ resp.\ $Z^\beta(t)$ are, conditionally on $\mathcal{F}_t$, independent copies of $X^{\lambda_\mu}$, each started with a single active resp.\ dormant particle at $0$, which are also independent of $\mathcal{F}_t$. 
 Thus we have 

 \begin{align*}
 \E\left[e^{-X^{\lambda_\mu}} \Big\vert \mathcal{F}_t\right]
   &=  \E\left[\prod_{\alpha \in I_t}e^{-Z^\alpha(t) \exp(\mu (M_t^\alpha+\lambda_\mu t))}\prod_{\beta \in J_t}e^{-Z^\beta(t)\exp(\mu(M_t^\beta+\lambda_\mu t))}\Big\vert \mathcal{F}_t\right]\\
      &= \prod_{\alpha \in I_t} \E_{(0, \boldsymbol{a})}\left[e^{-Z^{\alpha}(t) \exp(\mu (M_t^\alpha+\lambda_\mu t))}\Big\vert \mathcal{F}_t\right] \prod_{\beta \in J_t} \E_{(0, \boldsymbol{d})}\left[e^{-Z^{\beta}(t)\exp(\mu(M_t^\beta+\lambda_\mu t))}\Big\vert \mathcal{F}_t\right]\\
     &     = \prod_{\alpha \in I_t} \E_{(0, \boldsymbol{a})}\left[e^{-X^{\lambda_\mu} \exp(\mu x)}\right]\bigg\vert_{x=M_t^\alpha+\lambda_\mu t} \prod_{\beta \in J_t} \E_{(0, \boldsymbol{d})}\left[e^{-X^{\lambda_\mu}\exp(\mu x)}\right]\bigg\vert_{x=M_t^\beta+\lambda_\mu t} \\
     &=     Y_t^{\lambda_\mu},
 \end{align*}
where we also used \eqref{def-wave-function-2} for the last equality.
\end{proof}

We can also use the representation \eqref{def-wave-function} of the travelling waves 
 to establish the asymptotic behaviour of the waves at infinity.
In the following, we will repeatedly use the elementary estimate
\begin{equation}\label{eq:elementary-estimate}
\frac{z}{1+z}\le 1-e^{-z}\le z,\qquad z\ge0.
\end{equation}
\begin{proof}[Proof of Theorem \ref{prop:supercritical_travelling_wave_asymptotics}]
Let $\mu \in\, ]\mu^*, 0[$.
Then on the one hand, 
by \eqref{def-wave-function} and \eqref{eq:elementary-estimate} we clearly have
\begin{align*}
1- f_\mu(x)
=\E_{(x, \boldsymbol{a})}\left[1 - e^{-X^{\lambda_\mu} } \right]
\leq \E_{(x, \boldsymbol{a})}\left[ X^{\lambda_\mu}\right] = {\tt f}_\mu (x),\qquad x\in\R.
\end{align*}
On the other hand, we also obtain by \eqref{eq:elementary-estimate} and since ${\tt f}_\mu (x)={\tt f}_\mu (0)e^{\mu x}$ that
\begin{align*}
1- f_\mu(x) = \E_{(0, \boldsymbol{a})}\left[1 - e^{-X^{\lambda_\mu} e^{\mu x}} \right]
\geq \E_{(0, \boldsymbol{a})}\left[\frac{ X^{\lambda_\mu}e^{\mu x}}{1+ X^{\lambda_\mu}e^{\mu x}}\right]
=\E_{(0, \boldsymbol{a})}\left[\frac{ X^{\lambda_\mu}}{1+ X^{\lambda_\mu}e^{\mu x}}\right]\frac{{\tt f}_\mu (x)}{{\tt f}_\mu (0)},\qquad x\in\R.
\end{align*}
By monotone convergence, we have
\begin{align*}
\lim_{x \to \infty} \E_{(0, \boldsymbol{a})}\left[\frac{ X^{\lambda_\mu}}{1+ X^{\lambda_\mu}e^{\mu x}}\right] = \E_{(0, \boldsymbol{a})}[X^{\lambda_\mu}]={\tt f}_\mu(0),
\end{align*}
as desired. An analogous argument yields the second part of the statement.
\end{proof}

\noindent Next, we prove Theorem \ref{thm:supercritical_travelling_wave_asymptotics} (for model variant I) and show that for initial conditions whose decay behaviour agrees with the asymptotics of a super-critical travelling wave,  
the corresponding solutions to Equation \eqref{eq:FKPPDormancy} converge towards this travelling wave. 

\begin{proof}[Proof of Theorem \ref{thm:supercritical_travelling_wave_asymptotics}]

We are guided by \cite[Theorem 1.41]{C97}. Let $\mu \in\, ]\mu^*, 0[$. We first claim that by our assumptions we can find for $\varepsilon>0$ small and $x$ large enough the bounds \
 \begin{align}\label{eq-bound}
     \exp(-(1+\varepsilon ) {\tt f}_\mu (x))\leq u_0(x) \leq \exp(-(1-\varepsilon) {\tt f}_\mu(x)).
 \end{align}
Indeed, fix  $0<\varepsilon<1$. For the upper bound, pick $x$ large enough such that
 \begin{align*}
 	\abs{1-\frac{1-u_0(x)}{{\tt f}_\mu(x) }}&<\varepsilon ,
 \end{align*}
 which using \eqref{eq:elementary-estimate} implies
 \begin{align*}
  	u_0(x)&<1-(1-\varepsilon){\tt f}_\mu(x) \leq\exp(-(1-\varepsilon){\tt f}_\mu(x)).
 \end{align*}
 For the lower bound, 
 pick $x$ large enough such that
 \begin{align*}
 	\abs{1-\frac{1-u_0(x)}{{\tt f}_\mu(x) }}&<\varepsilon /4
 	 \end{align*}
 and also 
 \begin{align*}
 (1+\varepsilon){\tt f}_\mu(x)&<\varepsilon /2.
 \end{align*}
 Now, again using \eqref{eq:elementary-estimate}, by taking $\varepsilon$ small enough (and by increasing $x$ correspondingly so that the above bounds still hold true), we can ensure that 
 \begin{align*}
  1-\exp(-(1+\varepsilon){\tt f}_\mu(x)) &\geq \frac{(1+\varepsilon){\tt f}_\mu(x)}{1+(1+\varepsilon){\tt f}_\mu(x)} \geq \frac{1+\varepsilon}{1+\varepsilon /2} {\tt f}_\mu(x)\geq (1+\varepsilon/4){\tt f}_\mu(x) >1-u_0(x).
 \end{align*}
Rearranging then yields the lower bound in \eqref{eq-bound}. 
 Analogously, we have bounds
  \begin{align*}
     \exp(-(1+\varepsilon ) {\tt g}_\mu (x))\leq v_0(x) \leq \exp(-(1-\varepsilon) {\tt g}_\mu(x))
 \end{align*}
for $\varepsilon>0$ small and $x$ large enough. 

Now, again denoting by $L_t$ the position of the leftmost particle of the on/off BBM at time $t$, we know from \eqref{eq:Rt} that $\frac{L_t}{t}+\lambda_\mu\to\lambda_\mu-\lambda^*>0$ and thus
$$L_t+\lambda_\mu t \to \infty\qquad\text{a.s.}$$ 
as $t\to \infty$.
Hence, recalling the definition \eqref{eq:additive-martingale} of the additive martingale, by the above bounds
we obtain that
for all $\varepsilon>0$ small enough, almost surely for large enough $t$ 
 \begin{align}\label{eq:estimate-product}
     \exp\left(-(1+\varepsilon)X^{\lambda_\mu}_t \right)\leq \prod_{\alpha \in I_t} u_0(M_t^\alpha+\lambda_\mu t) \prod_{\beta \in J_t } v_0(M_t^\beta+\lambda_\mu t) \leq \exp\left(-(1-\varepsilon ) X^{\lambda_\mu}_t\right).
 \end{align}
 Taking limits and applying expectations (when starting from an active particle at $x \in \R$) yields 
 \begin{align*}
     \E_{(x,\boldsymbol{a})}\left[\exp(-(1+\varepsilon ) X^{\lambda_\mu} )\right] &\le
     \E_{(x,\boldsymbol{a})}\left[ \liminf_{t\to \infty} \prod_{\alpha \in I_t} u_0(M_t^\alpha+\lambda_\mu t) \prod_{\beta \in J_t } v_0(M_t^\beta+\lambda_\mu t) \right]\\
     &\leq \liminf_{t \to \infty}\E_{(x,\boldsymbol{a})}\left[\prod_{\alpha \in I_t} u_0(M_t^\alpha+\lambda_\mu t) \prod_{\beta \in J_t } v_0(M_t^\beta+\lambda_\mu t)\right]\\
     &= \liminf_{t \to \infty} u(t, x+\lambda_\mu t)\\
     &\leq \limsup_{t \to \infty} u(t, x+\lambda_\mu t)\\
     &= \limsup_{t \to \infty}\E_{(x,\boldsymbol{a})}\left[\prod_{\alpha \in I_t} u_0(M_t^\alpha+\lambda_\mu t) \prod_{\beta \in J_t } v_0(M_t^\beta+\lambda_\mu t)\right]\\
     &\leq \E_{(x,\boldsymbol{a})}\left[ \limsup_{t\to \infty} \prod_{\alpha \in I_t} u_0(M_t^\alpha+\lambda_\mu t) \prod_{\beta \in J_t } v_0(M_t^\beta+\lambda_\mu t) \right]\\
     &\le \E_{(x,\boldsymbol{a})}\left[\exp(-(1-\varepsilon ) X^{\lambda_\mu} )\right],
 \end{align*}
 where we used \eqref{eq:estimate-product}, Fatou's lemma and the duality 
 from Proposition \ref{prop:duality}.
 Since $\varepsilon >0$ was arbitrarily small, we get 
 \begin{align*}
     u(t,x+\lambda_\mu t) \to  f_\mu (x)
 \end{align*}
 as $t \to \infty$, and by a similar calculation starting the dual process from a dormant particle at $x \in \R$ we also get
  \begin{align*}
     v(t,x+\lambda_\mu t) \to  g_\mu (x),
 \end{align*}
 as desired.
\end{proof}

\subsection{Absence of travelling waves in the subcritical case $\lambda <\lambda^*$}
Finally, we prove Theorem \ref{thm:no_subcritical_monotone_waves} and show that there are no monotone travelling waves of speed strictly below $\lambda^*$.

\begin{proof}[Proof of Theorem \ref{thm:no_subcritical_monotone_waves}]
The proof follows along the lines of \cite[Section 3.4]{H99}. Assume that there exists a 
solution $(f^\lambda,g^ \lambda)$, increasing from $0$ to $1$, to 
 the travelling wave equation \eqref{eq:TravellingWaveI} with speed $\lambda < \lambda^*$. Then by Proposition \ref{prop: multmartingale}, the process
 \begin{align*}
    Y^\lambda_t = \prod_{\alpha \in I_t} f^\lambda(M_t^\alpha +\lambda t)  \prod_{\beta \in J_t} g^\lambda(M_t^\beta +\lambda t ),\qquad t \ge 0 
\end{align*}
is a martingale bounded in $[0,1]$. 
By the martingale convergence theorem, 
there exists some $[0,1]$-valued random variable $Y^\lambda$ such that 
\begin{align*}
    Y^\lambda_t \to Y^\lambda
\end{align*}
almost surely and in $L^1$ as $t \to \infty$. Hence, on the one hand we have 
\begin{align}
    f^\lambda(x)=\E_{(x,\boldsymbol{a})}\left[Y^\lambda_0\right]=\E_{(x,\boldsymbol{a})}\left[ Y^\lambda \right] \quad \mbox{ and } \quad 
    g^\lambda(x)=\E_{(x,\boldsymbol{d})}\left[Y^\lambda_0\right]=\E_{(x,\boldsymbol{d})}\left[ Y^\lambda \right]. \label{eq: representationoftravellingwave}
\end{align}
On the other hand, since $f^\lambda,g^\lambda$ take values in $[0,1]$,
\begin{align*}
    0\leq Y^\lambda_t =  \prod_{\alpha \in I_t} f^\lambda(M_t^\alpha +\lambda t)  \prod_{\beta \in J_t} g^\lambda(M_t^\beta +\lambda t ) \leq f^\lambda(L_t+\lambda t) \ind_I(t)+ g^\lambda (L_t +\lambda t)\ind_J(t),
\end{align*}
where again $L_t$ denotes the position of the leftmost particle at time $t$ and we denote by $I$ (resp. $J$) the random 
set of time points consisting of the intervals 
during which this leftmost particle is active (resp. dormant).
By \eqref{eq:Rt} we have $\frac{L_t}{t}+\lambda\to \lambda-\lambda^*<0$ almost surely and thus
$$L_t+\lambda t \to -\infty\qquad\text{a.s.}$$
From the fact that $(f^\lambda,g^\lambda)$ are travelling waves increasing from 0 to 1, we deduce that almost surely,
$$
f^\lambda(L_t+\lambda t) \ind_I(t) + g^\lambda(L_t +\lambda t)\ind_J(t) \to 0
$$
as $t \to \infty$. Hence, 
$Y^\lambda=0$
almost surely. But by Equation \eqref{eq: representationoftravellingwave} this implies  
$$
f^\lambda\equiv g^\lambda\equiv 0,
$$
which contradicts the properties of $(f^\lambda,g^\lambda)$. 
\end{proof}

\section{Discussion and outlook}\label{sec:discussion}

In what follows we aim at analyzing the impact of the introduction of dormancy on the wave speed, in both variants of the F-KPP equation with dormancy. We thus set $c=\tilde c, c'=\tilde c', {\tt s}=\tilde{\tt s}$ for the remainder of this section.

\subsection{Comparing the models}

For the natural parameter choices $c=c'={\tt s}=1$, Proposition \ref{prop:relation} shows that 
the critical wave speed of the classical F-KPP model dominates the one of model variant I (``seed bank model''), which in turns dominates the critical wave speed of model variant II (``spore model''). Interestingly, this monotonicity holds also with regard to the entire graphs of the respective speed functions. However, Figure \ref{fig:speed_functions_zoomed} shows that  a similar monotonicity does not hold for the {\em position} of the minima.  Indeed, while the position of the minimum of the speed function for the classical F-KPP equation (remarkably) coincides with the position of the minimum for variant II, this is not the case for variant I.

\begin{figure}[htbp]
    \centering
\begin{tikzpicture}
\begin{axis}[legend pos=outer north east, xlabel=$\mu$, ylabel=$\lambda_\mu$]

\addplot[blue,domain=-1.8:-1, samples=201] {-(-2 + x^2 + sqrt(20 + 4*x^2 + x^4))/(4*x)};
\draw[fill] (axis cs:{-1.19103,0.982416}) circle [radius=1.5pt] node[above left] {$ $};
\draw[dashed] (axis cs:{-2,0.982416}) -- (axis cs:{-0.9,0.982416});
\addlegendentry{seed bank model}

\addplot[green,domain=-1.8:-1, samples=201] {-(-2 + x^2 + sqrt(20 - 4*x^2 + x^4))/(4*x)};
\draw[fill] (axis cs:{-1.41421,1/sqrt(2)}) circle [radius=1.5pt] node[above left] {$ $};
\draw[dashed] (axis cs:{-2,1/sqrt(2)}) -- (axis cs:{-0.9,1/sqrt(2)});
\addlegendentry{spore model}

\addplot[red,domain=-1.8:-1, samples=201] {-(0.5*x + 1/x)};
\draw[fill] (axis cs:{-1.41421,sqrt(2)}) circle [radius=1.5pt] node[above left] {$ $};
\draw[dashed] (axis cs:{-2,sqrt(2)}) -- (axis cs:{-0.9,sqrt(2)});
\addlegendentry{classic F-KPP}

\draw[dashed] (axis cs:{-1.41421,2}) -- (axis cs:{-1.41421,0.5});
\draw[dashed] (axis cs:{-1.19103,2}) -- (axis cs:{-1.19103,0.5});

\end{axis}
\end{tikzpicture}
\caption{Position and value of the minima of the speed functions for all three models.}
\label{fig:speed_functions_zoomed}
\end{figure}
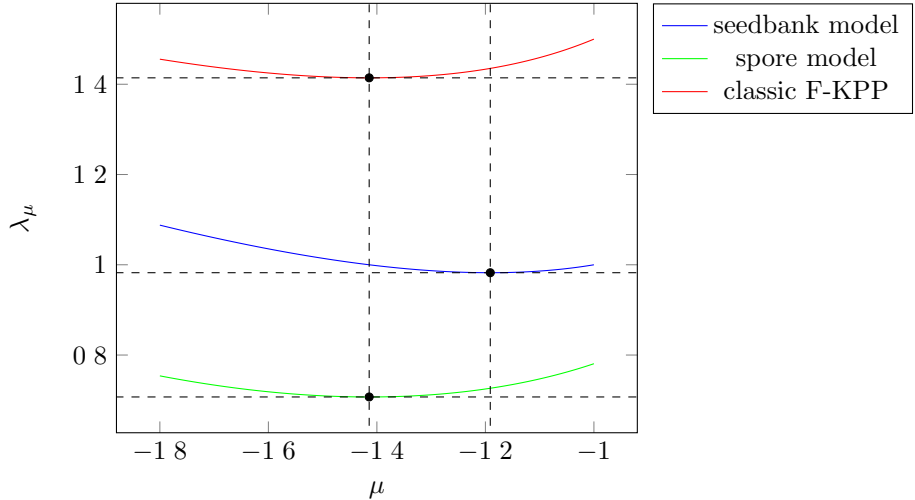

For the above unit parameters, the reduction of the speed of advance of an advantageous gene due to dormancy is more severe in the spore model  than in the seed bank model. On an intuitive level, this can be understood from their corresponding dual processes through Theorem \ref{thm:rightmost}. In on/off BBM of model variant I, individuals can both move and reproduce while being active. In particular, newborn particles can contribute to the spread of the population immediately after birth and are only slowed down due to switches into dormancy later on. In contrast, in the on/off BBM of model variant II, newly created actively reproducing individuals are initially non-moving and need to switch their type from $ \boldsymbol{a}$ to $ \boldsymbol{d}$ before they can be subject to dispersal, in turn preventing them from branching. 

However, the picture of the impact of dormancy over the whole parameter space is relatively complex, as we will see below.
Indeed, although the values of the critical wave speeds $\lambda^*$ and $\tilde{\lambda}^*$ of model variants I and II can be explicitly characterized, they are not just simple functions of the parameters $c, c'$ and ${\tt s}$, and in practice will often have to be computed numerically. To get a feeling for the effects of the different parameters on the critical wave speed in the various models, we highlight the following concrete scenarios.

\paragraph{Fixed transition rates $c, c'=1$, selection varying from $0$ to $\infty$:} This situation is depicted in the upper left panel in Figure \ref{fig:4}. The critical wave speeds in all three models grow as (apparently) concave functions from $0$ to $\infty$. However, the relative impact of dormancy in model I seems to become smaller and smaller in comparison to the classical F-KPP model. Interestingly, this is not the case for variant II, where the critical wave speed is always precisely half the speed of the classical model. In fact, we have that $$\lambda^{\ast, \tiny \rm classical} ({\tt s})= \sqrt{2{\tt s}}= 2  \tilde{\lambda}^*({\tt s})$$
for all ${\tt s} \in\ ]0, \infty[$. 
It would be interesting to find a quantitative intuitive argument for this precise relationship. One hint in this direction seems to be that if $c=c'$, then the asymptotic fraction of time that an individual spends in the active resp. dormant state is precisely $1/2$ (but this is of course also true for model I, which shows different behaviour).

\paragraph{Fixed selection rate ${\tt s}=1$, transition rates $c=c'$ simultaneously varying from 0 to $\infty$:}
This situation is depicted in the lower right panel in Figure \ref{fig:4}. 
The fixed relationship $c=c'$ ensures that the active and the dormant pool of individuals are of the same relative size. Again, the classical critical wave speed is precisely twice as large as the one of variant II, and they agree with the case ${\tt s}=1=c=c'$. This means in particular that the critical wave speed for the spore model is independent of the overall transitioning rate $c=c'$ in this case. But the dependence of model I on $c=c'$ is non-trivial: For small overall transition rates, transitions from the active population into the seed bank are rare. For the dual on/off BBM process of variant I this means that newly born particles will spend a long time in an actively moving and reproducing state. It is thus intuitively clear that the speed of the rightmost particle approaches the classical case as $c=c' \to 0$. In contrast, for very large switching rates, particles will almost immediately enter (and leave) the seed bank, with on average about half of the particles (including the newly created ones) being in a dormant state at any given time. These rapid fluctuations between active and dormant state slow the action of the Laplacian and the selection term by a factor of 1/2, explaining the reduced critical wave speed, which converges to the one of variant II.

\paragraph{Fixed ${\tt s}=c'=1$ while $c$  varies from $0$ to $\infty$:}
This situation is depicted in the upper right panel in Figure \ref{fig:4}, where the critical wave speed is represented as a function of $c$. Small values of $c$ correspond to a `small seed bank', that is, particles spend most of their time in an active state. It is thus not surprising that the behaviour of variant I is close to the classical case for small $c$, in contrast to variant II, where the dual process needs particles to be in the dormant state to engage in spatial motion.

Surprisingly, the critical wave speed as a function of $c$ seems to be uni-modal in variant II, hinting at a trade-off between selection and switching effects. Again this seems to be due to the mutually exclusive branching vs. motion character of the dual of variant II. More precisely, on the one hand diffusion is needed to colonize new areas and initiate travelling waves, but this only takes place while in the dormant non-reproducing state. On the other hand, the branching (selection) term in the active component contributes most to the linear speed of the rightmost particle, but it can only be effective if particles switch quickly into dormancy to find their way into new environments. It is thus reasonable to expect an equilibrium of sorts for these two effects that leads to a maximal wave speed. Our simulations suggest that this equilibrium is attained at $c=c'$, i.e.\ whenever switching to and from the dormant state happens at the same rate, which is again remarkable.

Finally, as $c \to \infty$, the critical wave speeds of both variants I and II approach 0. In this case, the effect of the dormant component dominates, and hence the effect of the selection becomes smaller and smaller.

\paragraph{Fixed ${\tt s}=c=1$ while $c'$  varies from $0$ to $\infty$:}
This situation is depicted in the lower left panel in Figure \ref{fig:4}. Note that large values of $c'$ mean that dormant particles in the dual processes `wake up quickly'.
For variant I the critical wave speed decreases as $c'$ decreases, since in the dual process, dormant particles cannot produce offspring. However, it is an interesting question whether 0 can actually be reached, and in fact a phase-transition seems to emerge: If $c'$ is close to 0, switching into dormancy (at rate $c$) amounts to an effective `killing', and branching can only happen during the initial active phase. The expected time in this initial phase is $1/c$, so that iff 
$$
{\tt s} \cdot \frac 1c > 1,
$$
the branching process will be (effectively) `super-critical', and the limit of the critical wave speed should stay above 0 as $c' \to 0$, whereas in the opposite case, one should see convergence to 0. This is consistent with the simulations underlying Figure \ref{fig:5}, for the values $c=1$ and  ${\tt s}=3/2$ (super-critical) and ${\tt s}=1/2$ (sub-critical). This is another example where arguments via the probabilistic dual process shed light on the behaviour of the original analytic system.

The critical wave speed of variant II again exhibits a uni-modal shape as the result of a trade-off that can be understood from the dual process with similar arguments as in the previous scenario.

\begin{figure}[htbp]
\label{fig:4}
\begin{center}
	\includegraphics[width=7.5cm]{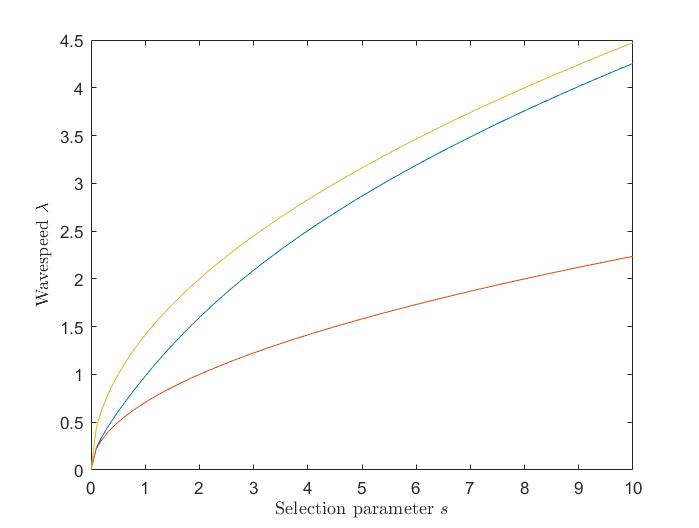}
	\includegraphics[width=7.5cm]{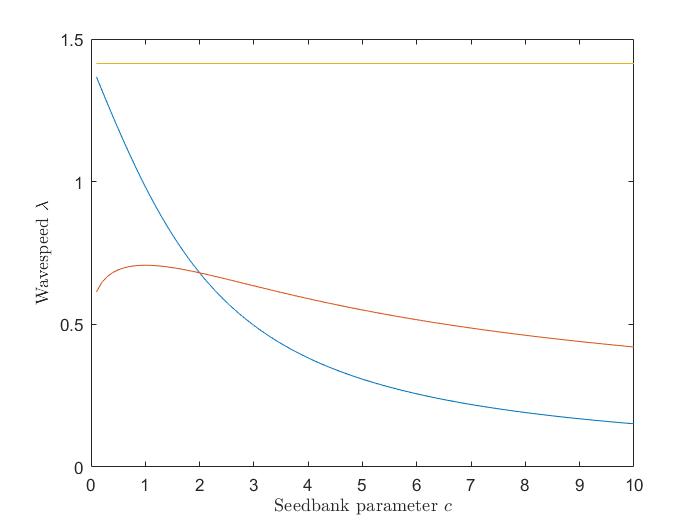}\\
	\includegraphics[width=7.5cm]{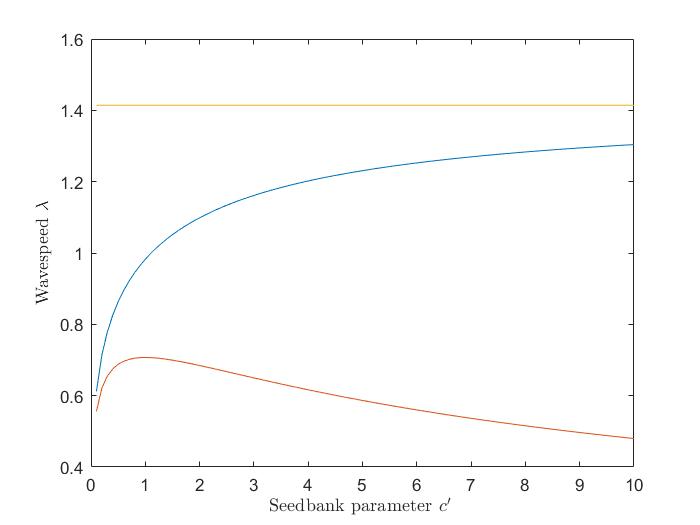}
	\includegraphics[width=7.5cm]{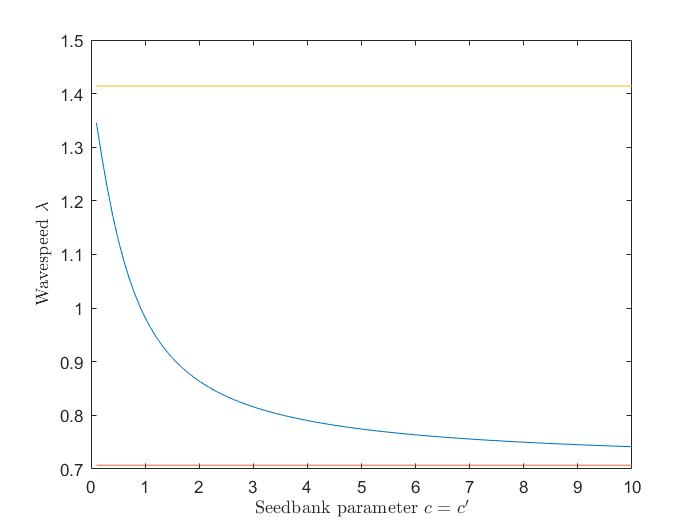}
		\caption{Critical wavespeed in classical FKPP (yellow), with seed bank (blue) and in the spore model (brown) with varying parameter ${\tt s},c,c',c=c'$ and all other parameters set to $1$.}
		\end{center}
\end{figure}

\begin{figure}[htbp]
\label{fig:5}
\begin{center}
	\includegraphics[width=7.5cm]{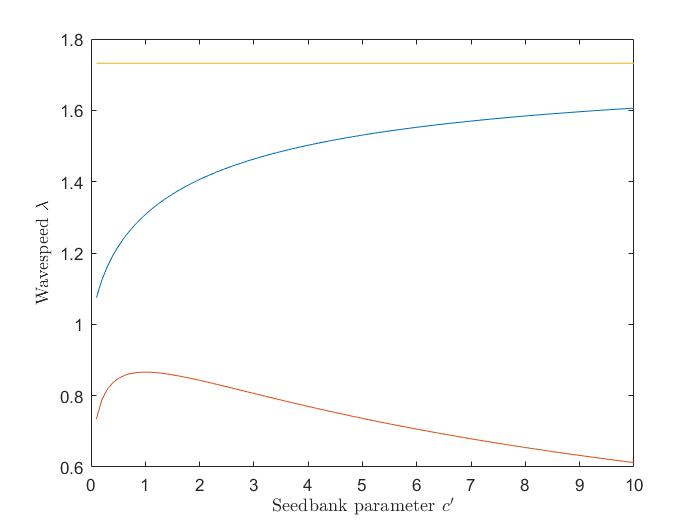}
	\includegraphics[width=7.5cm]{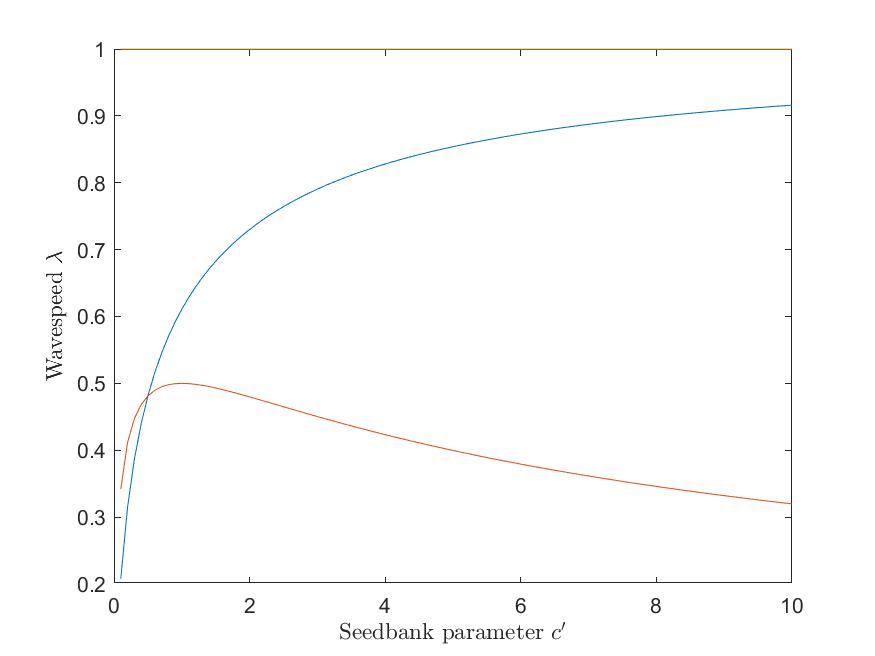}
		\caption{Critical wavespeed in classical FKPP (yellow), with seed bank (blue) and in the spore model (brown) with varying parameter $c'$ and ${\tt s}=3/2, c=1$ in the first and ${\tt s}=1/2, c=1$ in the second figure.}
		\end{center}
\end{figure}

\subsection{Related models}

Reaction-diffusion models similar to variant I also appear in theoretical neuroscience in the form of nerve-axon equations, see e.g.\  \cite{Evans1972_1}.
There, a rather general multi-component reaction diffusion system is investigated, where, similar to our systems, a spatial diffusion term is only present  in a single component. 
However, the author is specifically interested in travelling (nerve-)pulses rather than travelling waves, meaning that the `wave-shape $w$' is supposed to go to zero on both ends (i.e. $w(x) \to 0$ for $x \to +\infty$ \emph{and} $x \to -\infty$). The very general assumptions on the coupling terms in these papers have provided a framework for a rather large body of literature of related works, and many further variants of coupling terms have been discussed. The corresponding papers are typically focused on analytic and numerical aspects and indeed often do not exhibit duality relations to stochastic processes such as branching Brownian motion. Formulas for the speed of travelling waves of certain variants of the nerve-axon equations can for example be found in \cite{HP2000} (see p.\ 4) on contaminant transport.

The related paper \cite{XZ2005} considers a model for a cholera epidemic, where the dispersing individuals are interpreted as infectious bacteria and the non-dispersing individuals are considered as humans, acting as `incubators' for the reproduction of the bacteria  whose mobility is negligible. This model essentially looks the same as the spore model, however the coupling term in the second line of the spore model does not fulfill the required assumptions in \cite{XZ2005} on its second derivative\footnote{The corresponding formulation of the coupling term for the spore model (in the notation found in \cite{XZ2005}) would be $g(u) = u(1-u) + u = 2u -u^2$. Now \cite{XZ2005} considers the bistable case, which in particular requires the second derivative of the coupling term to take on positive and negative values, where the curvature of the coupling term for the spore model is strictly negative, $g''(u) = -2 < 0$.}. 

In \cite{ZWS2010} the authors use so-called speed index functions to investigate the wave speed. These are essentially Laplace-type transforms of an integral kernel. The wave speed in \cite{ZWS2010} then appears as a solution to an equation involving the speed index function and the model parameters. This is somehow reminiscent of the speed functions considered in the present paper. According to figures 5-7 and 11-17 of \cite{ZWS2010}, the dependence on the model parameters of the wave speed is monotone in contrast to the spore model, where we find a non-monotone dependence. 

\subsection{Open questions and future research}

Given the large body of work on the classical F-KPP equation and its variants, our results seem far from complete. In this section we briefly touch upon several open questions including the related technical difficulties, and outline 
some aspects for future research.

\paragraph{Convergence to the critical wave and its speed and shape.}
Note that our convergence theorem involving speed and shape of the wave (Theorem \ref{thm:supercritical_travelling_wave_asymptotics}) only covers the supercritical regime, and that our  results for the critical case (Theorem \ref{thm:critical_speed}) neither provide an asymptotic shape nor any convergence behaviour into the critical travelling wave.

This is stark contrast to the classical F-KPP equation, where one has  convergence of the solution, when started from suitably decaying initial conditions, into the critical travelling wave (see e.g.\ \cite{B83}). More precisely, there exists a process $(m_t)_{t \geq 0}$ such that
\begin{align*}
	u(t,x+m_t) \to w(x)
\end{align*}
as $t \to \infty$, where $u$ is the solution to the classical F-KPP equation started from e.g.\ Heaviside initial conditions and $w$ its travelling wave 
corresponding to the critical speed $\sqrt{2}$. Here, $m_t$ is chosen as the unique value such that 
\begin{align*}
u(t,m_t)=\frac{1}{2}
\end{align*}
for each $t \geq 0$. The asymptotics of $m_t$ have been studied in detail by e.g. McKean, Bramson and Roberts in \cite{McK75, B78, B83, Roberts13}, who obtained
$$
m(t)= \sqrt{2}t- \frac{3}{2\sqrt{2}} \log t + {\rm constant} + o(1).
$$ 
It would be desirable to obtain similar results for our versions of the model wirth dormancy, starting with the question whether one has the existence of a process $(m_t)_{t \geq 0}$, or rather two potentially different processes $(m^1_t)_{t \geq 0}$ and $(m^2_t)_{t \geq 0}$, such that
\begin{align*}
	u(t,x+m^1_t) \to f(x) \quad \mbox{ and } \quad  v(t,x+m^2_t) \to g(x)
\end{align*} 
as $t \to \infty$, alongside with the finer asymptotic results for their speeds. For coupled two-component systems as in our case this seems to have been an open question for the last couple of decades (see e.g. \cite[pp.\ 2,5]{E01}). 

Unfortunately, this also implies that many further results are currently inaccessible for higher order systems like ours, including e.g.\ the probabilistic representation of the travelling wave of Lalley and Sellke and related results regarding the shape of the critical wave such as \cite[Theorem 1]{LS87}.\\

\paragraph{Convergence of the additive martingale and existence of travelling waves in the critical case.}
Related to the above issues is the problem that our results regarding the convergence of additive martingales in Section \ref{sec:add_mart} cover everything except the critical case $\mu=\mu^*$. In the context of the classical F-KPP equation, this can be covered through a spine argument using the Girsanov Theorem and the fact that the quadratic variation of Brownian motion is deterministic (see e.g.\ \cite{Kyp04}). However, in the case of an on/off Brownian motion (without branching) the quadratic variation is truly probabilistic, making an application of the Girsanov Theorem difficult: For instance if $c=c'=1$, the on/off Brownian motion can be expressed by
\begin{equation*}
B_t^{\text{on/off}} = \int_0^t \underbrace{(N_s + 1) \!\!\!\mod 2}_{\in \{0,1\}} \, {\rm d}B_s,
\end{equation*}
\noindent where $N$ is a Poisson process with rate one and $B$ is a standard Brownian motion. Its quadratic variation is then given by the process
\begin{equation*}
[B^{\text{on/off}}]_t = \int_0^t ((N_s + 1) \!\!\!\mod 2)^2 \, {\rm d}s = \int_0^t ((N_s + 1) \!\!\!\mod 2) \, {\rm d}s
\end{equation*}
\noindent which is a random piece-wise linear function.

This also leads to a lack of existence results for travelling waves in the critical case $\lambda=\lambda^*$. An alternative approach through adoption of the stopping line theory by Chauvin \cite{C91} along the lines of \cite{H99} turns out to also be challenging for the multi-type case.

\paragraph{Uniqueness of travelling waves modulo translation.}
For the classical F-KPP equation it is well known that the monotone travelling waves 
from $0$ to $1$ are unique up to translations. A probabilistic approach may be found in \cite{H99} but relies on the convergence of the additive martingale for the critical case $\mu=\mu^*$.\\
We do however believe that the methods outlined in \cite[Theorem 1.41]{C97} in combination with the analytic methods from \cite[Section 3 (f)]{C97} may also yield the corresponding result for variant I of the F-KPP equation with dormancy (i.e.\ the seed bank model). Note however that small modifications will be necessary. For example, the phase plane formulation of the equation will involve the 3-dimensional vector $(u,v,u')$ instead of the 4-dimensional vector $(u,v,u',v')$, and the proof of  \cite[Lemma 3.2]{C97} will only yield exponential decay for the active part $(u-1,(u-1)')$ instead of the whole system $(u-1,v-1,(u-1)',(v-1)')$. These difficulties can be overcome, for example, by considering a delay reformulation of the system as in \cite{BHN19}. Since however we want to focus on probabilistic methods for the analysis of our models in this paper and the aforementioned methods require heavy use of phase plane analysis and differential equation theory, we refrain from providing the technical details here.

\paragraph{Properties of on/off branching Brownian motions.}

Branching Brownian motion has been a classical object of study in probability theory for more than 50 years, apparently beginning with \cite{M62}. In the last decade, it has experienced increased interest due to e.g.\ the construction and analysis of its extremal process (\cite{ABK13}, \cite{ABBZ13}).

It appears to be an interesting task to investigate related extremal properties of our two new variants of this model in the form of on/off branching Brownian motions.
This seems to extend the recent line of research on time-inhomogeneous or variable speed BBM
(\cite{BH15}, \cite{MZ16}) or in variable environments, e.g.\cite{MM19}
into a novel direction.

\section{Appendix} \label{sec:Appendix}

We recall the definition of the eigenvalue problem \eqref{eq:eigenvalue_problem} and of 
$\lambda_\mu=\lambda_\mu^+$ in \eqref{eq:lambda_mu}.

\begin{lemma}
    \label{lem:speedfunctionpositivity}
Let $\mu < 0$. Then we have $\lambda_\mu > 0$,
and the eigenvalue $-\mu\lambda_\mu$ in \eqref{eq:eigenvalue_problem} has a 
strictly positive eigenvector (which is unique up to a positive constant) explicitly given by
\begin{align*}
 \vec{d}(\mu)
 &= \begin{pmatrix}
1 \\ 1 - \frac{-\mu\lambda_\mu}{c'-\mu\lambda_\mu} 
\end{pmatrix}
>\begin{pmatrix}
0\\
0
\end{pmatrix}.
\end{align*}
\end{lemma}
\begin{proof}
The matrix
    \begin{equation*}
        B := \frac{1}{2}\mu^2A + Q + R = \begin{pmatrix}
            \frac{1}{2}\mu^2 - c + {\tt s} & c \\
            c' & -c'
        \end{pmatrix}
    \end{equation*}
    is quasipositive and irreducible, since $c,c'>0$. Thus
by a variant of the Perron-Frobenius-Theorem
(see \cite[Thm.\ 2.6, p.\ 46]{S1981}), it has a special (Perron-Frobenius) eigenvalue which is real and larger than the real part of all other eigenvalues. 
In view of this, the value $-\mu\lambda_\mu$ (being the larger of the two eigenvalues in \eqref{eq:eigenvalue_problem}) must be (real and) the Perron-Frobenius eigenvalue of $B$. 
By \cite[Thm.\ 2.6 (e)]{S1981}, $- \mu\lambda_\mu $ ist strictly positive if and only if for all vectors $\vec{y}\ge\vec{0}$, $\vec{y}\ne\vec{0}$ we have that at least one coordinate of $B\vec{y}$ is strictly positive. 
Now if $\vec{y}=(y_1,y_2)^\top$ with $y_1>y_2\ge0$, then 
$(B\vec{y})_2=c'(y_1-y_2)>0$, while if $y_2>y_1\ge0$ or $y_1=y_2>0$, then $(B\vec{y})_1=(\frac{1}{2}\mu^2+{\tt s})y_1+c(y_2-y_1)>0.$
Therefore we have $-\lambda_\mu \mu> 0$. Since $\mu < 0$, we get $\lambda_\mu > 0$.

The existence of a strictly positive eigenvector which is unique up to positive multiples also follows from the Perron-Frobenius-Seneta Theorem \cite[Thm.\ 2.6 (b)]{S1981}, but it can also be explicitly computed from the eigenvalue problem \eqref{eq:eigenvalue_problem}. 

\end{proof}

\begin{proof}[Proof of Proposition \ref{prop:speed_function}]
We first observe that
\begin{align}\label{eq:asymptotics-speed-function}
    \lim_{\mu \to -\infty} \lambda_{\mu}=\infty=    \lim_{\mu \to 0-} \lambda_{\mu} .
\end{align}
Indeed, the first equality follows directly from Equation \eqref{eq:lambda_mu}.
For the second, it is easily checked that 
\begin{align*}
\lim_{\mu\to0-}-2\mu\lambda_\mu=& {\tt s}-c'-c+\sqrt{c^2+2\,c\,c'-2\,c\,{\tt s}+({c'})^2+2\,c'\,{\tt s}+{\tt s}^2}\\
 & = {\tt s}-c'-c+ \sqrt{ (c-{\tt s})^2 + 2 c c' + (c')^2 + 2c'{\tt s} }\\
 &>0.
\end{align*}
Consequently, the speed function $\mu \mapsto \lambda_\mu$ (being continuous) has a global minimizer $\mu^*\in\,\ ]-\infty, 0[$. 

Now, assume there exists another (local) minimizer of the speed function $\lambda_\bullet$. 
Then due to \eqref{eq:asymptotics-speed-function},
there must exist distinct $\mu_1, \ldots ,  \mu_4 \in\, ]-\infty, 0[$ such that
\begin{align*}
    \lambda^\#\defeq\lambda_{\mu_1}=\lambda_{\mu_2}=\lambda_{\mu_3}=\lambda_{\mu_4}.
\end{align*}
But recall that $-\mu \lambda_\mu$ is an eigenvalue of 
$\frac{1}{2}\mu^2A + Q + R $, 
or equivalently 
$P(\mu,\lambda_\mu)=0$, where
$$P(\mu,\lambda) \defeq\det\left(\frac{1}{2}\mu^2A + Q + R+\mu \lambda I_2 \right). $$
Hence, on the one hand for fixed $\lambda$ the map $\mu \mapsto P(\mu, \lambda)$ is a polynomial of degree $3$, and on the other hand
$$P(\mu_1, \lambda^\#)=P(\mu_2, \lambda^\#)=P(\mu_3, \lambda^\#)=P(\mu_4, \lambda^\#)=0.$$
This is a contradiction. Consequently, $\mu^*$ is the unique local minimizer of the speed function.
\end{proof}

\begin{proof}[Proof of Proposition \ref{prop:relation}]
Note first that the respective speed functions for each model are given by
\begin{align*}
\lambda_\mu = -\frac{\sqrt{\frac{\mu^4}{4}+\mu^2+5}+\frac{\mu^2}{2}-1}{2\mu}, \quad 
\tilde \lambda_{ \mu} = -\frac{\sqrt{\frac{\mu^4}{4}-\mu^2+5}+\frac{\mu^2}{2}-1}{2\mu} \quad \mbox{and} \quad
\lambda^{\text{classical}}_\mu=-\frac{1+\frac{\mu^2}{2}}{\mu}.
\end{align*}
Then it is straightforward to show that 
\begin{align*}
\tilde \lambda_\mu< \lambda_\mu< \lambda^\text{classical}_\mu
\end{align*}
for all $\mu<0$, from which the result follows. 

\end{proof}

\subsection*{Acknowledgments}
The authors wish to thank Jay T. Lennon (Bloomington) for suggesting to analyze models in which dormancy and dispersal covary. 
This work has been supported by DFG IRTG 2544 Berlin-Oxford and by DFG under Germany’s Excellence Strategy – The Berlin Mathematics Research Center MATH+ (EXC-2046/1, project ID 390685689, BMS Stipend).

\newpage

\bibliographystyle{abbrv}
\bibliography{literature}

\begin{thebibliography}{10}

\bibitem{ABBZ13}
E.~A{\"\i}d{\'e}kon, J.~Berestycki, E.~Brunet, and Z.~Shi.
\newblock {Branching Brownian motion seen from its tip}.
\newblock {\em {Prob. Theory Rel. Field}}, 157:405--451, 2013.

\bibitem{ABK13}
L.-P. Arguin, A.~Bovier, and N.~Kistler.
\newblock {The extremal process of branching Brownian motion}.
\newblock {\em {Prob. Theory Rel. Fields}}, 157:535--574, 2013.

\bibitem{AT00}
S.~Athreya and R.~Tribe.
\newblock Uniqueness for a class of one-dimensional stochastic {PDE}s using
  moment duality.
\newblock {\em Ann. Probab.}, 28(4):1711--1734, 2000.

\bibitem{BHN19}
J.~Blath, M.~Hammer, and F.~Nie.
\newblock {The stochastic Fisher-KPP Equation with seed bank and on/off
  branching coalescing Brownian motion}.
\newblock {\em Stoch. PDE: Anal. Comp.}, 11:773--818, 2023.

\bibitem{B15}
A.~Bovier.
\newblock {\em Gaussian Processes on Trees: From Spin Glasses to Branching
  Brownian Motion}.
\newblock Cambridge Studies in Advanced Mathematics. Cambridge University
  Press, 2017.

\bibitem{BH15}
A.~Bovier and L.~Hartung.
\newblock {Variable Speed Branching Brownian Motion 1. Extremal Processes in
  the Weak Correlation Regime}.
\newblock {\em {ALEA}}, XII:261--291, 2015.

\bibitem{BH23}
A.~Bovier and L.~Hartung.
\newblock {The speed of invasion in an advancing population}.
\newblock {\em J. Math. Biol.}, 87(56):1--32, 2023.

\bibitem{B83}
M.~Bramson.
\newblock Convergence of solutions of the {K}olmogorov equation to travelling
  waves.
\newblock {\em Mem. Amer. Math. Soc.}, 44(285):iv+190, 1983.

\bibitem{B78}
M.~D. Bramson.
\newblock Maximal displacement of branching {B}rownian motion.
\newblock {\em Comm. Pure Appl. Math.}, 31(5):531--581, 1978.

\bibitem{BC14}
M.~Buoro and S.~M. Carlson.
\newblock Life-history syndromes: integrating dispersal through space and time.
\newblock {\em Ecology Letters}, (17):756--767, 2014.

\bibitem{C97}
A.~Champneys, S.~Harris, J.~Toland, J.~Warren, and D.~Williams.
\newblock Algebra, analysis and probability for a coupled system of
  reaction-diffusion equations.
\newblock {\em Philosophical Transactions: Physical Sciences and Engineering},
  350(1692):69--112, 1995.

\bibitem{C91}
B.~Chauvin.
\newblock {Product Martingales and Stopping Lines for Branching Brownian
  Motion}.
\newblock {\em Ann. Probab.}, 19(3):1195--1205, 07 1991.

\bibitem{E01}
U.~Ebert and W.~{van Saarloos}.
\newblock Front propagation into unstable states: universal algebraic
  convergence towards uniformly translating pulled fronts.
\newblock {\em Physica D: Nonlinear Phenomena}, 146(1):1--99, 2000.

\bibitem{Evans1972_1}
J.~W. {Evans}.
\newblock {Nerve axon equations. I: Linear approximations}.
\newblock {\em {Indiana Univ. Math. J.}}, 21:877--885, 1972.

\bibitem{F37}
R.~A. Fisher.
\newblock The wave of advance of an advantageous gene.
\newblock {\em Ann. Eugenics}, 1937.

\bibitem{F91}
M.~{Freidlin}.
\newblock {Coupled reaction-diffusion equations}.
\newblock {\em {Ann. Probab.}}, 19(1):29--57, 1991.

\bibitem{GHO20}
A.~Greven, F.~den Hollander, and M.~Oomen.
\newblock Spatial populations with seed-bank: well-posedness, duality and
  equilibrium.
\newblock {\em Electron. J. Probab.}, 27(18):1--88, 2022.

\bibitem{H99}
S.~C. Harris.
\newblock {Travelling-waves for the FKPP equation via probabilistic arguments}.
\newblock {\em Proceedings of the Royal Society of Edinburgh: Section A
  Mathematics}, 129(3):503--517, 1999.

\bibitem{HP2000}
D.~{Hilhorst} and M.~A. {Peletier}.
\newblock {Convergence to travelling waves in a reaction-diffusion system
  arising in contaminant transport}.
\newblock {\em {J. Differ. Equations}}, 163(1):89--112, 2000.

\bibitem{INW68a}
N.~Ikeda, M.~Nagasawa, and S.~Watanabe.
\newblock Branching {M}arkov processes. {I}.
\newblock {\em J. Math. Kyoto Univ.}, 8:233--278, 1968.

\bibitem{INW68b}
N.~Ikeda, M.~Nagasawa, and S.~Watanabe.
\newblock Branching {M}arkov processes. {II}.
\newblock {\em J. Math. Kyoto Univ.}, 8:365--410, 1968.

\bibitem{INW69}
N.~Ikeda, M.~Nagasawa, and S.~Watanabe.
\newblock Branching {M}arkov processes. {III}.
\newblock {\em J. Math. Kyoto Univ.}, 9:95--160, 1969.

\bibitem{KPP37}
A.~Kolmogorov, N.~Petrovsky, and N.~Piscounov.
\newblock Etude de l' \'equation de la diffusion avec croissance de la
  quantit\'e de mati\`ere et son application \`a un probl\`eme biologique.
\newblock {\em Moscow Univ. Math. Bull.}, (1):1---25, 1937.

\bibitem{Kyp04}
A.~E. Kyprianou.
\newblock {Travelling wave solutions to the K-P-P equation : alternatives to
  Simon Harris' probabilistic analysis}.
\newblock {\em Annales de l'I.H.P. Probabilit\'es et statistiques},
  40(1):53--72, 2004.

\bibitem{LS87}
S.~P. Lalley and T.~Sellke.
\newblock A conditional limit theorem for the frontier of a branching
  {B}rownian motion.
\newblock {\em Ann. Probab.}, 15(3):1052--1061, 1987.

\bibitem{LJ11}
J.~Lennon and S.~Jones.
\newblock Microbial seed banks: The ecological and evolutionary implications of
  dormancy.
\newblock {\em Nature reviews. Microbiology}, 9:119--30, 02 2011.

\bibitem{LHWB21+}
T.~Lennon, F.~den Hollander, M.~Wilke~Berenguer, and J.~Blath.
\newblock Principles of seed banks: Complexity emerging from dormancy.
\newblock {\em Nat. Comm.}, (12), 2021.

\bibitem{MZ16}
P.~Maillard and O.~Zeitouni.
\newblock {Slowdown in branching Brownian motion with inhomogeneous variance}.
\newblock {\em Annales de l'Institut Henri Poincar{\'e}, Probabilit{\'e}s et
  Statistiques}, 52(3):1144 -- 1160, 2016.

\bibitem{MM19}
B.~Mallein and P.~Mi\l{}o\`s.
\newblock Maximal displacement of a supercritical branching random walk in a
  time-inhomogeneous random environment.
\newblock {\em Stochastic Process. Appl.}, 129(9):3239--3260, 2019.

\bibitem{McK75}
H.~P. McKean.
\newblock Application of {B}rownian motion to the equation of
  {K}olmogorov-{P}etrovskii-{P}iskunov.
\newblock {\em Comm. Pure Appl. Math.}, 28(3):323--331, 1975.

\bibitem{M62}
J.~E. {Moyal}.
\newblock {Multiplicative population chains}.
\newblock {\em {Proc. R. Soc. Lond., Ser. A}}, 266:518--526, 1962.

\bibitem{N88}
J.~Neveu.
\newblock Multiplicative martingales for spatial branching processes.
\newblock In E.~{\c{C}}inlar, K.~L. Chung, R.~K. Getoor, and J.~Glover,
  editors, {\em Seminar on Stochastic Processes, 1987}, pages 223--242, Boston,
  MA, 1988. Birkh{\"a}user Boston.

\bibitem{Roberts13}
M.~I. Roberts.
\newblock {A simple path to asymptotics for the frontier of a branching
  Brownian motion}.
\newblock {\em The Annals of Probability}, 41(5):3518 -- 3541, 2013.

\bibitem{S1981}
E.~Seneta.
\newblock {\em Non-negative Matrices and Markov Chains}.
\newblock Springer Series in Statistics. Springer New York : Imprint: Springer,
  New York, NY, 2nd ed. edition, 1981.

\bibitem{S88}
T.~Shiga.
\newblock Stepping stone models in population genetics and population dynamics.
\newblock In {\em Stochastic processes in physics and engineering ({B}ielefeld,
  1986)}, volume~42 of {\em Math. Appl.}, pages 345--355. Reidel, Dordrecht,
  1988.

\bibitem{S64}
A.~V. Skorohod.
\newblock Branching diffusion processes.
\newblock {\em Teor. Verojatnost. i Primenen.}, pages 492--497, 1964.

\bibitem{Wat67}
S.~Watanabe.
\newblock Limit theorem for a class of branching processes.
\newblock In J.~Chover, editor, {\em Markov processes and potential theory.},
  Proc. Sympos. Math. Res. Center Madison Wis. 1967, pages 205--232. Wiley, New
  York, 1967.

\bibitem{WLL19}
N.~I. Wisnoski, M.~A. Leibold, and J.~T. Lennon.
\newblock Dormancy in metacommunities.
\newblock {\em Am. Nat.}, (194):135--151, 2019.

\bibitem{XZ2005}
D.~{Xu} and X.-Q. {Zhao}.
\newblock {Erratum to ``Bistable waves in an epidemic model''}.
\newblock {\em {J. Dyn. Differ. Equations}}, 17(1):219--247, 2005.

\bibitem{ZWS2010}
L.~{Zhang}, P.-S. {Wu}, and M.~A. {Stoner}.
\newblock {Influence of sodium currents on speeds of traveling wave fronts in
  synaptically coupled neuronal networks}.
\newblock {\em {Physica D}}, 239(1-2):9--32, 2010.

\end{thebibliography}

\end{document}